\documentclass{amsart}

\usepackage{amsfonts,amssymb,amsmath,amsbsy}
\usepackage{times}
\usepackage{a4wide}
\newtheorem{theorem}{Theorem}[section]
\newtheorem{lemma}[theorem]{Lemma}

\newcommand{\R}{{\mathbb R}}

\newcommand{\N}{{\mathbb N}}

\newcommand{\E}{\mathbb{E}}
\newcommand{\PP}{\mathbb{P}}

\newcommand{\HH}{{\mathcal H}}

\numberwithin{equation}{section}

\begin{document}

\title[Intrinsic volumes of random polytopes]
{Intrinsic volumes of random polytopes with vertices on the
boundary of a convex body}

\author[K.\ B\"or\"oczky]{K\'aroly J. B\"or\"oczky}
\address{Alfr\'ed R\'enyi Institute of Mathematics, 
Hungarian Academy of Sciences, 
 PO Box 127, H--1364 Budapest, Hungary 
and 
Department of Geometry, Roland E\"otv\"os University, Budapest, Hungary
}
\curraddr{Universitat Polit\'ecnica de Catalunya, BarcelonaTech, Spain}
\email{carlos@renyi.hu}
\thanks{The first author was supported by
OTKA grant 75016, and by the EU Marie Curie FP7 IEF grant GEOSUMSETS}

%    author two information
\author[F.\ Fodor]{Ferenc Fodor}
\address{Department of Geometry, University of Szeged, 
Aradi v\'ertan\'uk tere 1, H-6720 Szeged, Hungary and Department of Mathematics and Statistics, University of Calgary, 2500 
University Dr. N.W., Calgary, Alberta, Canada, T2N 1N4
}
\curraddr{}
\email{fodorf@math.u-szeged.hu}
\thanks{The second author was supported by Hungarian OTKA grants 68398 and 75016 and
by the J\'anos Bolyai Research Scholarship of the Hungarian Academy
of Sciences.}

%    author three information
\author[D.\ Hug]{Daniel Hug}
\address{Karlsruhe Institute of Technology,  
Department of Mathematics,  
D-76128 Karlsruhe, 
Germany}
\curraddr{}
\email{daniel.hug@kit.edu}
\thanks{}

%    \subjclass is required.


\date{\today}

\dedicatory{}

%    Abstract is required.
\begin{abstract}
Let $K$ be a convex body in $\R^d$, let $j\in\{1, \ldots, d-1\}$, and let
$\varrho $ be a positive and continuous probability
density function with respect to the $(d-1)$-dimensional Hausdorff 
measure on the boundary $\partial K$ of $K$.
Denote by $K_n$ the convex hull of $n$ points chosen randomly and
independently from $\partial K$ according to the probability distribution
determined by $\varrho$.
For the case when $\partial K$ is a $C^2$ submanifold of $\R^d$ with 
everywhere positive Gauss curvature, M. Reitzner proved an
asymptotic formula
for the expectation of the difference of the $j$th intrinsic volumes
of $K$ and $K_n$, as $n\to \infty$.
In this article, we extend this result to the case when the only
condition on $K$ is that a ball rolls freely in $K$. 
\end{abstract}

\maketitle

\section{Introduction}

Random polytopes in Euclidean space $\R^d$ can be defined 
in various ways. If $x_1,\ldots,x_n$ are $n$ random points 
sampled from a given convex body $K\subset\R^d$, then the 
convex hull of these random points yields a random polytope 
that has been studied extensively in the literature. The present 
focus is on a related though different model of a random polytope 
that has not been explored to the same extent. Instead of choosing 
the points from all of $K$, we sample random points from the boundary 
of $K$. The convex hull of these points then provides a model 
of a random polytope that will be considered here.
Our main focus is on the convergence of the expectation of geometric 
functionals (intrinsic volumes) of such a random polytope. 
The main result, stated in 
Theorem \ref{main}, extends previous work by relaxing the 
regularity assumptions on $K$. This is a nontrivial task, since 
the speed of convergence depends in a crucial way on the boundary 
structure, in particular on the (generalized) curvatures, of $K$. 
The present approach refines   
arguments that have recently been developed in \cite{BFH} 
to establish first order results for the aforementioned 
model of a random polytope, and it combines geometric and probabilistic ideas.

Before stating our results explicitly, we provide the required background 
and notation. Our basic setting is  
the $d$-dimensional Euclidean space $\R^d$, $d \ge 2$, 
with scalar product $\langle\cdot,\cdot\rangle$ and norm $\|\cdot\|$. 
By $\mathcal{H}^{j}$ we denote the $j$-dimensional Hausdorff measure, 
where $\mathcal{H}^{d}$ is simply called the volume 
$V_d$. Let $B^j$ be the unit ball of ${\mathbb R}^j$ with center at the 
origin, and let $S^{j-1}$ be its boundary. Then we write 
$\alpha_j=\mathcal{H}^{j}(B^j)$ for the $j$-dimensional volume of $B^j$, 
and hence $\mathcal{H}^{j-1}(S^{j-1})=j\alpha_j$ is the surface content of $B^j$. 
The relative boundary of a compact convex set $C\subset\R^d$ 
is denoted by $\partial C$. Finally, the convex hull of
subsets $X_1,\ldots,X_r$ and points $z_1,\ldots,z_s$ is denoted by
$[X_1,\ldots,X_r,z_1,\ldots,z_s]$. 

Throughout the following, $K$ is a convex body (compact convex set) with interior
points in $\R^d$; for notions of convexity we follow the 
monographs by Schneider \cite{Sch93} or  Gruber \cite{Gru07}. The boundary 
of $K$ is denoted by $\partial K$. 
We say that $\partial K$ is twice differentiable in the generalized sense at
a boundary point $x\in\partial K$ if there exists a positive semi-definite 
quadratic form $Q$ on $\R^{d-1}$, the so called second fundamental form,
with the following property: If $K$ is positioned in such a way that $x=o$
and $\R^{d-1}$ is a support hyperplane of $K$,  then in a neighborhood
of the origin $o$,  $\partial K$ is the graph of a convex function $f$ defined on a
$(d-1)$-dimensional ball around $o$ in  $\R^{d-1}$ satisfying
\begin{equation}
\label{diff}
f(z)=\mbox{$\frac12$}\,Q(z)+o(\|z\|^2),
\end{equation}
as $z\to o$.
According to a classical result of Alexandrov
(see P.M. Gruber \cite{Gru07} or R. Schneider \cite{Sch93}),
 $\partial K$ is twice differentiable in the generalized
sense at $\HH^{d-1}$ almost all points $x\in\partial K$. 
Such boundary points are also called normal boundary points. 
We write $k_1(x), \ldots , k_{d-1}(x)$ for the (generalized) principal
curvatures of $\partial K$ at $x\in \partial K$, which are just the eigenvalues of $Q$.
Furthermore, $H_j(x)$ denotes
the normalized $j$th elementary symmetric function of the principal
curvatures of $\partial K$ at the normal boundary point $x$. Here the dependence 
of this function on $K$ is not made explicit. Thus, for $j\in\{1,\ldots,d-1\}$, we have 
$$
H_j(x)=\binom{d-1}{ j}^{-1}
\sum_{1\leq i_1< \cdots< i_j\leq d-1}k_{i_1}(x)\cdots k_{i_j}(x),
$$
and this definition is supplemented by $H_0(x):=1$. 
In particular, $H_{d-1}(x)$ is the Gaussian curvature and $H_1(x)$ is
the mean curvature of $\partial K$ at $x$.
We say that $\partial K$ is $C^k_+$,
for some $k\geq 2$, if $\partial K$ is a $C^k$ submanifold of $\R^d$ and its
Gaussian curvature is positive everywhere.

The intrinsic volumes $V_j(K)$, $j=0, \ldots , d$, of a convex body $K\subset \R^d$
can be introduced as coefficients of the Steiner formula
$$V_d(K+\lambda B^d)=\sum_{j=0}^{d}\lambda^{d-j}\alpha_{d-j}V_j(K),$$
where $K+\lambda B^d$ is the Minkowski sum of $K$ and the ball $\lambda B^d$
of radius $\lambda\geq 0$. In particular, $V_d$ is the volume
functional, $V_0(K)=1$, $V_1$ is proportional to the mean width and
$V_{d-1}$ is a multiple of the surface area. Alternately, 
intrinsic volumes can be obtained as mean projection volumes.
Specifically, for $j=1, \ldots d-1$, it is well-known that
$$
V_j(K)=\frac{\binom{d}{ j}\alpha_d}{\alpha_j\alpha_{d-j}}
\int_{{\mathcal L}_j^d}V_j(K| L)\,\nu_j(dL),
$$
where ${\mathcal L}_j^d$ is the Grassmannian of all $j$-dimensional
linear subspaces of $\R^d$ equipped with the (unique) Haar probability
measure $\nu_j$ and, for $L\in {\mathcal L}_j^d$, $K| L$ denotes the
orthogonal projection of $K$ onto $L$. Here, $V_j(K| L)$ is just
the $j$-dimensional volume (Lebesgue measure) of $K| L$.

We say that a ball rolls freely in a convex body $K\subset\R^d$ if there
exists some $r>0$ such that any $x\in \partial K$ lies on the boundary
of some Euclidean ball $B$ of radius $r$ with $B\subset K$. The existence of
a rolling ball is equivalent to saying that the exterior unit normal
is a Lipschitz map on $\partial K$ (see D. Hug \cite{Hug2000}). In particular,
W. Blaschke observed that if $\partial K$ is $C^2$, then $K$ has a
rolling ball (see D. Hug \cite{Hug2000} or K. Leichtweiss \cite{L1998}). 
In turn, we say that $K$ rolls freely in a ball of radius $R>0$ if
any $x\in \partial K$  lies on the boundary
of some Euclidean ball $B$ of radius $R$ with $K\subset B$.

In this paper, we shall consider the following probability model.
Let $K$ be a convex body with a rolling ball of radius $r$. Let
$\varrho$ be a continuous, positive probability density function
defined on $\partial K$; throughout this paper  
this density is always considered with respect to the 
boundary measure on $\partial K$. Select the points $x_1, \ldots , x_n$
randomly and independently from $\partial K$ according to the
probability distribution determined by $\varrho$. The convex hull
$K_n:=[x_1, \ldots , x_n]$ then is a random polytope inscribed in $K$.
We are going to study the expectation of intrinsic volumes of
$K_n$. In order to indicate the dependence on the probability density $\varrho$, 
we write $\PP_\varrho$ to denote the probability
of an event in this probability space and $\E_\varrho$ to denote the expected value. 
For a convex body $K$,
the expected value $\E_\varrho(V_j(K_n))$ of the $j$-th intrinsic volume of
$K_n$ tends to $V_j(K)$ as $n$ tends to infinity. It is clear that the 
asymptotic behavior of
$V_j(K)-\E_\varrho(V_j(K_n))$ is determined by the shape of
the boundary of $K$. In the case when the boundary of $K$ is a $C_+^2$ submanifold of 
$\R^d$, this asymptotic behavior was described
by M. Reitzner \cite{Reitz2002}.

\begin{theorem}[Reitzner, 2002]
\label{reitznertheo}
Let $K$ be a convex body in $\R^d$ with $C^2_+$ boundary,
and let $\varrho$ be a continuous, positive probability
density function on $\partial K$. Denote
by $\E_\varrho (V_j(K_n)), j=1, \ldots, d$, the expected $j$-th intrinsic volume
of the convex hull of $n$ random points on $\partial K$ chosen
independently and according to the density function $\varrho$. Then
\begin{equation}
\label{reitzner}
V_j(K)-\E_\varrho(V_j(K_n))\sim
c^{(j, d)}\int_{\partial K}\varrho(x)^{-\frac{2}{d-1}}
H_{d-1}(x)^{\frac{1}{d-1}}H_{d-j}(x)\, \HH^{d-1}(dx)\cdot n^{-\frac{2}{d-1}}
\end{equation}
as $n\to\infty$, where the constant $c^{(j, d)}$ only depends on $j$ and
the dimension $d$.
\end{theorem}

For $j=d$, that is in the case of the volume functional,
C. Sch\"utt and E. Werner \cite{ScW03} extended \eqref{reitzner}
to any convex body $K$ such that
a ball of radius $r$ rolls freely in $K$ and, in addition, 
$K$ rolls freely in a ball of radius $R$, 
for some $R>r>0$. The latter assumption of $K$ rolling freely inside a ball 
implies a uniform positive lower bound for the principle curvatures of $\partial K$ 
whenever they exist. 
 They also calculated the constant $c^{(d,d)}$ explicitly, that is
$$
c^{(d, d)}=\frac{(d-1)^{\frac{d+1}{d-1}}\Gamma(d+1+\frac{2}{d-1})}
{2(d+1)![(d-1)\alpha_{d-1}]^{\frac{2}{d-1}}}.
$$
Moreover, C. Sch\"utt and E. Werner \cite{ScW03} showed that
for fixed $K$, the minimum of the integral expression
in \eqref{reitzner} is attained for the probability
density function
$$
\varrho_0(x)=\frac{H_{d-1}(x)^{\frac{1}{d+1}}}
{\int_{\partial K}H_{d-1}(x)^{\frac{1}{d+1}}\; \HH^{d-1}(dx)}.
$$

Our main goal is to extend Theorem~\ref{reitznertheo} to 
the case where $K$ is only assumed to have a rolling ball, for
all  $j=1,\ldots,d$. In particular, the Gauss curvature is allowed 
to be zero on a set of positive boundary measure. 
More explicitly, we shall prove

\begin{theorem}
\label{main}
The asymptotic formula \eqref{reitzner} holds if 
$K$ is a convex body in $\R^d$ in which a ball rolls freely.
\end{theorem}

The present method of proof for Theorem \ref{main} is different from the one used by  
Reitzner \cite{Reitz2002} or Sch\"utt and Werner \cite{ScW03}. It is 
inspired by the arguments from our previous paper \cite{BFH} concerning 
random points chosen from a convex body, however, the case 
of random points chosen from the boundary is more delicate. 

Examples show that in general the condition that a ball rolls freely inside $K$ 
cannot be dropped in Theorem~\ref{main}. General bounds are provided in the 
following theorem.

\begin{theorem}
\label{meanwidthbounds}
Let $K$ be a convex body in $\R^d$, and let $\varrho$ be a continuous, positive 
probability density function on $\partial K$. Then 
there exist positive constants  
$c_1,c_2$, depending on $K$ and $\varrho$, such that
for any $n\geq d+1$, 
$$
c_1n^{-\frac2{d-1}}\leq
\E_\varrho (V_1(K)-V_1(K_{n}))\leq c_2n^{-\frac1{d-1}}.
$$
The lower bound is of
optimal order if $K$ has a rolling ball,
and the upper bound is of optimal order,
if $K$ is a polytope.
\end{theorem}

For comparison, let us review the main known results
about the convex hull $K(n)$ of $n$ points chosen randomly,
independently and uniformly from $K$. In the case where 
a ball rolls freely inside $K$, the analogue of Theorem~\ref{main} 
is established in K. B\"or\"oczky Jr., L. M. Hoffmann and D. Hug \cite{BoHoHu2008}.
 For the case of the volume functional and an arbitrary convex body $K$,
C. Sch\"utt \cite{Sch94} proved (see
K.J. B\"or\"oczky, F. Fodor, D. Hug \cite{BFH} for some corrections and an extension) 
that
$$
\lim_{n\to\infty}n^{\frac{2}{d+1}}(V_d(K)-\E(V_d(K(n)))=
c_d V_d(K)^{\frac{2}{d+1}}\int_{\partial K}H_{d-1}(x)^{\frac{1}{d+1}}\, \HH^{d-1}(dx),
$$
 where the constant $c_d>0$ only depends on the dimension $d$ and is explicitly known.
 Concerning the order of approximation, we have
 \begin{equation}
 \label{1}
\gamma_1 n^{-2/(d+1)} < V_1(K)-\E V_1(K(n)) < \gamma_2n^{-1/d},
\end{equation}
\begin{equation}
\label{2}
\gamma_3 n^{-1}\ln ^{d-1}n < V_d(K)-\E V_d(K(n)) < \gamma_4n^{-2/(d+1)},
\end{equation}
where $\gamma_1,\ldots,\gamma_4>0$ are constants that may depend on $K$. 
The inequalities (\ref{1}) are due to R. Schneider \cite{Sch87}, and (\ref{2}) is due to
 I. B\'{a}r\'{a}ny and D. Larman \cite{BL88}. The orders are best possible, being
 attained in (\ref{1})(left) and (\ref{2})(right) by sufficiently smooth bodies,
 and in (\ref{1})(right) and (\ref{2})(left) by polytopes.
 
 The proof of Theorem \ref{main} is given in the following three sections. 
 In Section 2, we rewrite the difference $V_j(K)-\E_\varrho(V_j(K_n))$ in 
 an integral geometric way. The inner integral involved in this integral 
 geometric description is extended over the projection $K|L$ of $K$ to $L$, where 
 $L$ is a $j$-dimensional linear subspace. Then we show that up to an error 
 term of lower order the main contribution 
 comes from a neighborhood of the (relative) boundary $\partial (K|L)$ of 
 $K|L$ with respect to $L$, where this neighborhood 
 is shrinking at a well-defined speed $t(n)$ as $n\to\infty$. Further application of 
 an integral geometric decomposition then shows that the proof boils down to 
 determining the limit
$$
\lim_{n\to\infty}
\int_0^{t(n)}n^{\frac{2}{d-1}}\langle y,u(y)
\rangle\PP_\varrho\left( y_t\notin K_n|L \right)\,dt,
$$
where $y\in\partial (K|L)$ and $x$ is a normal boundary point of $K$ with 
$y=x|L$. The case where the Gauss curvature of $K$ at $x$ is zero is 
treated directly. In Section 3, we deal with the case of positive 
Gauss curvature. In a first step, we choose a reparametrization of the integral 
which relates the parameter $t$ to the probability content $s$ of that part of the boundary 
of $K$ near $x$ that is cut off by a cap determined by the parameter $t$. This reparametrization 
has the effect of extracting the relevant geometric information from $K$.  
What remains to be shown is that the transformed integrals are essentially independent 
of $K$ and yield the same value for the unit ball with the uniform probability density 
on its boundary. 
This latter step is divided into two lemmas in Section 3. Whereas both lemmas have analogues in 
our previous work \cite{BFH}, the present arguments are more delicate and the second lemma 
has to be established by a reasoning different from the one in \cite{BFH}. The proof 
is then completed in Section 4, where, in addition to the previous steps,  
a very special case of Theorem \ref{reitznertheo} 
is employed ($K$ being the unit ball) as well as an integral geometric lemma 
from \cite{BoHoHu2008}. The final section is devoted to a proof of Theorem \ref{meanwidthbounds}.

\section{General Estimates}

In order to prove Theorem~\ref{main}, we start by rewriting 
$V_j(K)-\E_\varrho(V_j(K_n))$ in an integral geometric form. For this, 
we use Kubota's formula
and Fubini's theorem to obtain
\begin{eqnarray}
\label{basic}
&&V_j(K)-\E_\varrho(V_j(K_n))\nonumber\\
&&=\int_{\partial K}\ldots\int_{\partial K}
\left(V_j(K)-V_j(K_n)\right)
\prod_{i=1}^{n}\varrho(x_i)\,\HH^{d-1}(dx_1)\ldots
\HH^{d-1}(dx_n)\nonumber\\
&&=\frac{\binom{d}{ j}\alpha_d}{\alpha_j\alpha_{d-j}}
\int_{\partial K}\ldots\int_{\partial K}\int_{{\mathcal L}_j^d}
\left(V_j(K| L)-V_j(K_n| L)\right)\nonumber\\
&&\qquad\qquad\times\prod_{i=1}^{n}\varrho(x_i)\,
\nu_j(dL)\,\HH^{d-1}(dx_1)\ldots
\HH^{d-1}(dx_n)\nonumber\\
&&=\frac{\binom{d}{j}\alpha_d}{\alpha_j\alpha_{d-j}}\int_{{\mathcal L}_j^d}
\int_{L}\int_{\partial K}\ldots\int_{\partial K}
{\mathbf 1}\left\{y\in K| L \text{ and } y\not\in K_n|L\right\} \nonumber\\
&&\qquad\qquad\times
\prod_{i=1}^{n}\varrho(x_i)\,
\HH^{d-1}(dx_1)\ldots \HH^{d-1}(dx_n)\HH^{j}(dy)\,\nu_j(dL)\nonumber\\
&&=\frac{\binom{d}{ j}\alpha_d}{\alpha_j\alpha_{d-j}}
\int_{{\mathcal L}_j^d}\int_{K| L}\PP_\varrho(y\not\in K_n| L)
\,\HH^{j}(dy)\,\nu_j(dL).
\end{eqnarray}

Now we  introduce some geometric tools.
If $K$ has a rolling ball of radius $r$, then so does
$K| L$ for any $L\in {\mathcal L}_j^d$.
Furthermore, $\partial K$ has a unique
outer unit normal vector $u(x)$ at each boundary point $x\in\partial K$.
If $L\in {\mathcal L}_j^d$, $y\in\partial (K| L)$
and $x\in K$ such that $y=x| L$, then $x\in\partial K$ and
the outer unit normal of $\partial (K| L)$ at $y$ is equal
to $u(x)$.

Since the statement of the theorem is translation invariant,
we may assume
that
\begin{equation}
\label{rR}
rB^d\subset K\subset RB^d
\end{equation}
for some $R>0$. For $t\in(0,1)$, let
$K_t:=(1-t)K$, and for $x\in\partial K$, let $x_t:=(1-t)x$.
Similarly, $(K|L)_t:=(1-t)(K| L)$ and $y_t:=(1-t)y$ for
$y\in\partial (K|L)$. 

For $x\in\partial K$ and $t\in (0,1)$, let
$$
x^*_t:=x-\langle tx,u(x)\rangle u(x).
$$
If $t\in(0,\frac{r}R)$, then \eqref{rR} implies that
\begin{equation}\label{nach4}
tr\leq \langle x-x^*_t,u(x)\rangle=\langle x-x_t,u(x)\rangle<r.
\end{equation}
The existence of a rolling ball at $x$ yields that
if $t\in(0,\frac{r}R)$, then
\begin{equation}
\label{x*}
x^*_t+r\sqrt{t}(u(x)^\bot\cap B^d)\subset K.
\end{equation}
On the other hand, we have
\begin{equation}
\label{x*x}
\|x^*_t-x_t\|<Rt.
\end{equation}

For real functions $f$ and $g$ defined on the same space, we write
$f\ll g$ or $f=O(g)$ if there exists a positive constant $\gamma$, depending
only on $K$ and $\varrho$, such that
$|f|\leq \gamma\cdot g$.

We shall use the notion
of a ``coordinate corner''.
Given an orthonormal basis in a linear $i$-dimensional subspace $L$,
the corresponding $(i-1)$-dimensional coordinate planes cut $L$ into
$2^i$ convex cones, which we call coordinate corners (with respect to $L$
and the given basis). In the following, we write $\gamma_1,\gamma_2,\ldots $
for positive constants which merely depend on $K$ and $\varrho$.

Let us estimate the probability that $o\not\in K_n$.
There exists a constant
$\gamma_1>0$ such that the probability content of each
of the parts of $\partial K$
contained in one of the $2^d$ coordinate corners of $\R^d$ is at least
$\gamma_1$. Now if $o\not\in K_n$, then $o$ can be strictly separated
from $K_n$ by a hyperplane. It follows that
 $\{x_1,\ldots,x_n\}$ is disjoint from one of
these coordinate corners, and hence
\begin{equation}
\label{except}
\PP (o\notin K_n)\leq 2^d(1-\gamma_1)^n.
\end{equation}
This fact will be used, for instance, in the proof of the subsequent lemma. 
In the following, for $x\in\R^d$ we use the shorthand notation  
$\R_+x:=\{\lambda x:\lambda\ge 0\}$.

\begin{lemma}
\label{genup}
There exist constants $\delta,{\gamma}_2\in(0,1)$,
depending on $K$ and $\varrho$,
such that if $L\in {\mathcal L}_j^d$,
$y\in \partial (K| L)$ and $t\in (0,\delta)$, then
$$
\PP_\varrho\left(y_t\not\in K_n| L \right)\ll
\left(1-{\gamma}_2t^{\frac{d-1}2}\right)^n.
$$
\end{lemma}
\begin{proof} Let $y\in \partial (K| L)$ and $x\in \partial K$ be
such that $y=x| L$.
Let $\Theta'_1,\ldots,\Theta'_{2^{d-1}}$ be the coordinate corners
with respect to some basis vectors
in $u(x)^\bot$. In addition, for $i=1,\ldots,2^{d-1}$
and $t\in(0,1)$, let
$$
\Theta_{i,t}=\partial K\cap\left(x_t+\left[\Theta'_i,\R_+x\right]
\right).
$$
Since $\varrho$ is positive and continuous, we have
$$
\int_{\Theta_{i,t}}\varrho(x)\,\HH^{d-1}(dx)\ge \gamma_3{\HH}^{d-1}(\Theta_{i,t}).
$$
If $y_t\not\in K_n| L$ and $o\in K_n$, then there exists a
$(j-1)$-dimensional affine plane $H_L$ in $L$ through $y_t$, bounding the
halfspaces $H_L^-$ and $H_L^+$ in $L$, for which
$K_n| L\subset H_L^-$. Now, if $L^\bot$ is the orthogonal complement
of $L$ in $\R^d$, then $H:=H_L+L^\bot$ is a
hyperplane in $\R^d$ with the property that $x_t\in H$
and $K_n\subset H^-:=H_L^-+L^\bot$.
Furthermore, $\Theta_{i,t}\subset H^+:=H_L^++L^\bot$
for some $i\in\{1,\ldots,2^{d-1}\}$, because $o\in K_n\subset H^-$.
 Therefore
$$
\PP_\varrho\left( y_t\not\in K_n| L , o\in K_n\right)\le
\sum_{i=1}^{2^{d-1}}
\left(1-\gamma_3 \HH^{d-1}(\Theta_{i,t})\right)^n.
$$
Combining  \eqref{x*} and \eqref{x*x}, we deduce the existence of 
a constant $\gamma_4>0$ such that
if $t\leq \gamma_4$, then the orthogonal projection
of $\Theta_{i,t}$ into $u(x)^\bot$
 contains a translate of $\Theta_i'\cap (r/2)\sqrt{t}B^d$, and therefore
$$\HH^{d-1}(\Theta_{i,t})\ge \gamma_5 t^{\frac{d-1}2}$$
for $i=1,\ldots,2^{d-1}$.
In turn, we obtain
\begin{equation}
\label{eqa}
\PP_\varrho\left( y_t\not\in K_n| L , o\in K_n\right)\ll
\left(1-\gamma_6 t^{\frac{d-1}2}\right)^n.
\end{equation}
On the other hand, if $o\not\in K_n|L$, then
\eqref{except} holds. Combining this with
(\ref{eqa}), we conclude the proof of the lemma.
\end{proof}

Subsequently, the estimate of Lemma \ref{genup} will be used, for instance, 
to restrict the
domain of integration (cf.\ Lemma \ref{closetoK}) and
to justify an application of Lebesgue's
dominated convergence theorem (see \eqref{limitform}).
For these applications,
we also need that
if $x\in\partial K$ and $c>0$ satisfies
$\bar\omega:=c\delta^{\frac{d-1}2}<1$,
then
\begin{equation}
 \label{Gamma}
\int_0^\delta\left(1-ct^{\frac{d-1}2}\right)^n\, dt
= c^{\frac{-2}{d-1}}\frac2{d-1}
\int_0^{\bar\omega}s^{\frac2{d-1}-1}(1-s)^n\,ds
\ll  c^{\frac{-2}{d-1}}\cdot n^{\frac{-2}{d-1}},
\end{equation}
where we use that $(1-s)^{n}\le e^{-ns}$ for $s\in[0,1]$ and $n\in\N$.

The next lemma will allow us to decompose integrals in a suitable way.
We write $u(y)$ to denote
the unique exterior unit normal to $\partial (K| L)$ at $y\in\partial (K| L)$.
It will always be clear from the context whether
we mean the exterior unit normal at a point $x\in\partial K$ or
at a point $y\in\partial (K| L)$.

\begin{lemma}
\label{integration}
If \, $0\le t_0 < t_1<\delta$ and
$h:K| L\to [0,\infty]$ is a measurable function, then
\begin{align*}
&\int_{(K| L)_{t_0}\setminus (K| L)_{t_1}}\mathbb{P}_\varrho
\left( x\notin K_n| L\right ) h(x) \, \HH^j(dx)\\
&\qquad\qquad=
\int_{\partial (K| L)}\int_{t_0}^{t_1}(1-t)^{j-1}\mathbb{P}_\varrho
\left( y_t\notin K_n| L \right)
\langle y,u(y)\rangle h(y_t)\, dt\, \HH^{j-1}(dy).
\end{align*}
\end{lemma}

\begin{proof}
The set $\partial (K| L)$ is a  $(j-1)$-dimensional submanifold of $L$ of class $C^1$, and the map
$$
T:\partial (K| L)\times (t_0,t_1)\to
{\rm int}(K| L)_{t_0}\setminus (K| L)_{t_1},
\mbox{ \ }(y,t)\mapsto y_t,
$$
is a $C^1$ diffeomorphism with Jacobian
$JT(y,t)=(1-t)^{j-1}\langle y,u(y)\rangle\ge 0$.
Thus the assertion
follows from Federer's area/coarea theorem (see \cite{Federer69}).
\end{proof}

In the following, we use the abbreviation $t(n):=n^{\frac{-1}{d-1}}$.

\begin{lemma}
\label{closetoK}
Let $1\leq j\leq d-1$. Then we have
$$
\int_{{\mathcal L}_j^d}\int_{(K|L)_{t(n)}}
\PP_\varrho\left(y\not\in K_n| L \right)\,\HH^{j}(dy)\,\nu_j(dL)=
o\left(n^{\frac{-2}{d-1}}\right).
$$
\end{lemma}

\begin{proof} Let $\delta>0$ be chosen such that it satisfies the
conditions of Lemma~\ref{genup}. We may assume that
$n$ is large enough to satisfy $t(n)<\delta$ and $n\ge (\gamma_2)^2$. 
First, we treat that part of the integral which extends over the subset 
$(K|L)_{\delta}$ of $(K|L)_{t(n)}$. 

Let $\omega:=\delta r$. Then \eqref{nach4} yields
\begin{equation}
\label{omegadef}
\langle x-x_{\delta},u(x)\rangle\geq \omega\mbox{ \ \ for \ }
x\in\partial K.
\end{equation}
There exists a constant $\gamma_7>0$ such that the probability measure of
$(x+\frac{\omega}2\;B^d)\cap\partial K$ is at least
$\gamma_7$ for all $x\in\partial K$. We choose a maximal set
$\{z_1,\ldots,z_m\}\subset\partial K$ such that $\|z_i-z_l\|\geq \frac{\omega}2$
for $i\neq l$.

For $L\in{\mathcal L}_j^d$, let
$y\in(K|L)_{\delta}$. If $y\not\in K_n| L$, then
there exist a
hyperplane $H$ in $\R^d$ and a half space 
$H^-$ bounded by $H$ such that $y\in H$,
$H$ is orthogonal to $L$,
and $K_n\subset \text{int}(H^-)$. 
Choose $x\in \partial K$
such that $u(x)$ is an exterior unit normal to $H^-$.
Since $H$ intersects
$K_{\delta}$, we have $\langle x-y,u(x)\rangle\ge \omega$
by (\ref{omegadef}). Now there exists some $i\in\{1,\ldots,n\}$
with $\|x-z_i\|\leq \frac{\omega}2$, and hence 
$\{x_1,\ldots,x_n\}\subset \text{int}(H^-)$ yields
that $\{x_1,\ldots,x_n\}$ is disjoint from
$z_i+\frac{\omega}2\;B^d$. In particular, we have
\begin{equation}
\label{yinKLdelta}
\PP_\varrho\left(y\not\in K_n| L \right)
\leq m(1-\gamma_7)^n.
\end{equation}
Next let  $y\in \partial (K|L)$. If $t\in(t(n),\delta)$, then
Lemma~\ref{genup} yields
\begin{equation}
\label{yttn0}
\PP_\varrho\left(y_t\not\in K_n|L \right)\ll
\left(1-{\gamma}_2n^{-\frac{1}{2}}\right)^n
<e^{-{\gamma}_2n^{\frac{1}{2}}}\ll n^{\frac{-3}{d+1}}.
\end{equation}
In particular, writing $I$ to denote
the integral in Lemma~\ref{closetoK}, we obtain from
 Lemma~\ref{integration}, \eqref{yinKLdelta} and  \eqref{yttn0} that
\begin{eqnarray*}
I&\ll&
\int_{{\mathcal L}_j^d}\int_{(K|L)_{\delta}}
\PP_\varrho\left(y\not\in K_n|L \right)\,\HH^{j}(dy)\,\nu_j(dL)+\\
& &+\int_{{\mathcal L}_j^d}\int_{t(n)}^{\delta}\int_{\partial (K|L)}
\PP_\varrho\left(y_t\not\in K_n|L \right)\,\HH^{j-1}(dy)\,dt\,\nu_j(dL)\\
&\ll&m(1-\gamma_7)^n+\int_{{\mathcal L}_j^d}\int_{\partial (K|L)}
n^{\frac{-3}{d-1}}\,\HH^{j-1}(dy)\,\nu_j(dL)=
o\left(n^{\frac{-2}{d-1}}\right),
\end{eqnarray*}
which is the required estimate.
\end{proof}

It follows by applying \eqref{basic}, Lemma~\ref{closetoK}
and  Lemma~\ref{integration}, in this order,
that
\begin{align*}
&\lim_{n\to\infty}n^{\frac{2}{d-1}}(V_j(K)-\E_\varrho(V_j(K_n)))\\
&=\frac{\binom{d}{ j}\alpha_d}{\alpha_j\alpha_{d-j}}\lim_{n\to\infty}n^{\frac{2}{d-1}}
\int_{{\mathcal L}_j^d}\int_{K| L}\PP_\varrho(y\notin K_n| L)
\,\HH^{j}(dy)\,\nu_j(dL)\\
&=\frac{\binom{d}{j}\alpha_d}{\alpha_j\alpha_{d-j}}
\lim_{n\to\infty}n^{\frac{2}{d-1}}
\int_{{\mathcal L}_j^d}\int_{(K|L)\setminus (K| L)_{t(n)}}
\PP_\varrho(y\notin K_n| L)\,\HH^{j}(dy)\,\nu_j(dL)\\
&=\frac{\binom{d}{ j}\alpha_d}{\alpha_j\alpha_{d-j}}
\lim_{n\to\infty}
\int_{{\mathcal L}_j^d}\int_{\partial (K| L)}\int_{0}^{t(n)}
n^{\frac{2}{d-1}}\PP_\varrho(y_t\notin K_n| L)
(1-t)^{j-1}\langle y, u(y)\rangle\, dt\,\HH^{j-1}(dy)\,\nu_j(dL).
\end{align*}

We deduce from Lemma~\ref{genup} and (\ref{Gamma}) that if
$n>n_0$, $L\in {\mathcal L}_j^d$ and $y\in\partial (K|L)$, then
$$
\int_0^{t(n)}
n^{\frac2{d-1}} \PP_\varrho\left(y_t\not\in K_n|L \right)
\langle y,u(y)\rangle(1-t)^{j-1}\, dt\ll C,
$$
where $n_0$ and $C$ depend on $K$ and $\varrho$.
Therefore, we may apply Lebesgue's dominated convergence theorem,
and thus we conclude
\begin{equation}
\label{limitform}
\lim_{n\to\infty}n^{\frac{2}{d-1}}(V_j(K)-\E_\varrho(V_j(K_n)))=
\frac{\binom{d}{ j}\alpha_d}{\alpha_j\alpha_{d-j}}
\int_{{\mathcal L}_j^d}
\int_{\partial (K| L)}J_\varrho(y, L)\, \HH^{j-1}(dy)\,\nu_j(dL),
\end{equation}
where, for $L\in {\mathcal L}_j^d$ and $y\in\partial (K|L)$, we have
\begin{equation}
\label{Jdef}
J_\varrho(y, L):=\lim_{n\to\infty}
\int_0^{t(n)}n^{\frac{2}{d-1}}\langle y,u(y)
\rangle\PP_\varrho\left( y_t\notin K_n|L \right)\,dt.
\end{equation}
Subsequently, we shall inspect this limit more closely. 
In a first step, we shall consider those points $y\in \partial (K|L)$ for which 
there is a normal boundary point $x\in \partial K$ with $y=x|L$ and $H_{d-1}(x)=0$.

\begin{lemma}
\label{zerocurv}
Let $L\in {\mathcal L}_j^d$, and let $y\in\partial (K|L)$. If
$x\in \partial K$ is a normal boundary point of $K$ with $y=x|L$ and $H_{d-1}(x)=0$,
then $J_\varrho(y, L)=0$.
\end{lemma}

\begin{proof} Let $x\in\partial K$ be a normal boundary point 
with $y=x|L$ and $H_{d-1}(x)=0$. 
First, we show the existence
of a decreasing function $\varphi$
on $(0,\frac{r}R)$ with
$\lim_{t\to 0^+}\varphi(t)=\infty$ satisfying
\begin{equation}
\label{eqainf}
\PP_\varrho\left( y_t\not\in K_n| L \right)\leq 2^{d-1}
\left(1-\varphi(t) t^{\frac{d-1}2}\right)^n.
\end{equation}
In the following, we always assume that $t>0$ is sufficiently small, that is 
$n$ is sufficiently large, so that all expressions that arise are well defined.  
Let $v_1,\ldots,v_{d-1}$ be an orthonormal basis
in $u(x)^\bot$ such that these vectors are 
principal directions of curvature of $K$ at $x$ and such that 
the curvature is zero in the direction
of $v_1$. In addition, let $\Theta'_1,\ldots,\Theta'_{2^{d-1}}$
be the coordinate corners
in $u(x)^\bot$, and, for $i=1,\ldots,2^{d-1}$
and $t\in(0,1)$, let $\Theta_{i,t}=\partial K\cap\left(x_t+\left[\Theta'_i,\R_+x\right]
\right)$ as before. The continuity of $\varrho$ yields that 
$$
\int_{\Theta_{i,t}}\varrho(x)\,\HH^{d-1}(dx)\gg {\HH}^{d-1}(\Theta_{i,t}).
$$

Since the curvature is zero in the direction
of $v_1$, there exists a function $\psi$ on $(0,\frac{r}R)$ with
$\lim_{t\to 0^+}\psi(t)=\infty$ satisfying
$$
x^*_t-\psi(t)\sqrt{t}v_1\in K \mbox{ \ and \ }
x^*_t+\psi(t)\sqrt{t}v_1\in K.
$$
Combining  \eqref{x*} and \eqref{x*x}, we deduce the existence
of a decreasing function $\tilde{\varphi}$
on $(0,\frac{r}R)$ with
$\lim_{t\to 0^+}\tilde{\varphi}(t)=\infty$ satisfying
$$
\int_{\Theta_{i,t}}\varrho(x)\,\HH^{d-1}(dx)\geq \tilde{\varphi}(t)t^{\frac{d-1}2},
$$
for $i=1,\ldots,2^{d-1}$.

First, we assume that $y_t\not\in K_n| L$ and $o\in K_n$.
In particular, then we also have $x_t\not\in K_n$, and hence
there exists a hyperplane $H$ through $x_t$
such that $K_n$ lies on one side of $H$.
Since $o\in K_n$, it follows that $H$ separates
$K_n$ from some $\Theta_{i,t}$, and  therefore
\begin{equation}
\label{eqainf0}
\PP_\varrho\left( y_t\not\in K_n| L , o\in K_n\right)\leq 2^{d-1}
\left(1-\tilde{\varphi}(t) t^{\frac{d-1}2}\right)^n.
\end{equation}
On the other hand, if $o\not\in K_n|L$, then
\eqref{except} holds. Combining this with
\eqref{eqainf0}, we conclude \eqref{eqainf}.
In turn, we deduce from \eqref{Gamma} that
$$
J_\varrho(y, L)\ll\lim_{n\to\infty}n^{\frac{2}{d-1}}
\int_0^{t(n)}
(1-\varphi(t(n)) t^{\frac{d-1}2})^n\,dt\ll
\lim_{n\to\infty}\varphi(t(n))^{\frac{-2}{d-1}}=0.
\mbox{ \ }
$$
\end{proof}

In the next section, we study the more difficult case of boundary 
points with positive Gauss curvature.

\section{Normal boundary points and caps}

Let $L\in {\mathcal L}_j^d$, and let $y\in\partial (K|L)$ be such that
$y=x|L$ for some (uniquely determined) normal boundary point $x\in\partial K$ with $H_{d-1}(x)>0$.
We keep $x$ and $y$
fixed throughout this section.
First, we reparametrize $x_t$ and $y_t$ in terms
of the probability measure of the
corresponding cap of $\partial K$.
Using this reparametrization,
we show that $J_\varrho(y, L)$ essentially depends
only on the random points near $x$
(see Lemma~\ref{pointsneary}), and then in a second step we
pass from the case of a general convex body $K$ to the case of a Euclidean ball.

For  $t\in (0,1)$, we consider
the hyperplane $H(x,t):=
\{z\in \R^d:\langle u(x),z\rangle=\langle u(x),x_t\rangle\}$, the half-space 
$H^+(x,t):=\{z\in \R^d:\langle u(x),z\rangle\ge\langle u(x),x_t\rangle\}$, and
the cap $C(x,t):=K\cap H^+(x,t)$ 
whose bounding hyperplane is $H(x,t)$. Next we
reparametrize $x_t$ in terms of the induced
probability measure of the cap $C(x,t)$; namely,
$$
\tilde{x}_s:=x_t\quad \mbox{ \ and \ }\quad \tilde{y}_s:=y_t,
$$
where, for a given sufficiently small $s\ge 0$, the parameter $t\ge 0$ is uniquely 
determined by the equation
\begin{equation}\label{sdef}
s=\int_{C(x,t)\cap\partial K}\varrho(w)\, \HH^{d-1}(dw).
\end{equation}
Note that $s$ is a strictly increasing and continuous function of $t$. 
We further define
\begin{equation}
\label{capK}
\widetilde{C}(x,s)=C(x,t)\qquad\text{and}\qquad
\widetilde{H}(x,s)=H(x,t),
\end{equation}
where again, for given $s$, the parameter $t$ is determined by \eqref{sdef}.
Observe that $\partial K\cap H^+(x,t)=\partial K\cap C(x,t)$. 
Subsequently, we explore the relation between $s$ and $t$. 
Let $f:u(x)^\perp\to [0,\infty]$ 
be a convex function such that the restriction of the map 
$$
F:u(x)^\perp\to\R^d,\qquad z\mapsto x+z-f(z)u(x),
$$
to a neighborhood of $o$ parametrizes $\partial K$ in a neighborhood of $x$. 
Moreover, we consider the transformations
$$
\Pi:\R^d\to u(x)^\perp,\qquad y\mapsto y-x-\langle y-x,u(x)\rangle u(x),
$$
and
$$
T:u(x)^\perp\times\R\to u(x)^\perp\times\R,\qquad (z_1,\ldots,z_{d-1},\alpha)\mapsto
(\sqrt{k_1}z_1,\ldots,\sqrt{k_{d-1}}z_{d-1},\alpha),
$$
where $u(x)^\perp$ is considered to be a subset of $u(x)^\perp\times\{0\}$ 
and $k_i=k_i(x)$, $i=1,\ldots,d-1$, are the principle 
curvatures of $\partial K$ at $x$. Then we obtain
\begin{eqnarray*}
&&\int_{\partial K\cap H^+(x,t)}\varrho(w)\,\HH^{d-1}(dw)\\
&&=\int_{\Pi(\partial K\cap H^+(x,t))}\varrho(F(z))
\sqrt{1+\|\nabla f(z)\|^2}\,\HH^{d-1}(dz)\\
&&=\int_{T(\Pi(\partial K\cap H^+(x,t)))}
\varrho(F\circ T^{-1}(z))\sqrt{1+\|\nabla f(T^{-1}(z))
\|^2}\,H_{d-1}(x)^{-1/2}\,\HH^{d-1}(dz).
\end{eqnarray*}
Let $\overline{K}:=T(K-x)+x$, and hence 
$T(\Pi(\partial K\cap H^+(x,t)))=\Pi(\partial\overline{K}\cap H^+(x,t))$.
If $\overline{f}$ is defined for $\overline{K}$ as $f$ is defined for $K$, and
$$
\overline{\varrho}(w):=\varrho(F\circ T^{-1}\circ\Pi(w)),\qquad
g(w):=\frac{\sqrt{1+\|\nabla f(T^{-1}(\Pi(w)))\|^2}}{\sqrt{1+\|\nabla \overline{f}(\Pi(w))\|^2}},
$$
for $ w\in \partial \overline{K}\cap H^+(x,t)$, then we obtain
$$
\int_{\partial K\cap H^+(x,t)}\varrho(w)\,\HH^{d-1}(dw)
=H_{d-1}(x)^{-1/2}\int_{\partial \overline{K}\cap H^+(x,t)}\overline{\varrho}(w)g(w)\,\HH^{d-1}(dw).
$$
Next we put $H(r):=x-ru(x)+u(x)^\perp$ and denote by $n_{\overline{K}}(w)$ the exterior unit normal of
$\overline{K}$ at $w\in\partial \overline{K}$. Since (cf.~the notes for Section 1.5 (2) in \cite{Sch93})
$$
\overline{f}(z)=\frac{1}{2}\|z\|^2+o(\|z\|^2), \qquad \|\nabla  \overline{f}(z)\|=\|z\|+o(\|z\|),\qquad n_{\overline{K}}(w)=
\frac{\nabla \overline{f}(\bar w)+u(x)}{\sqrt{1+\|\nabla 
\overline{f}(\bar w)\|^2}}
$$
with $\bar w:=\Pi (w)$ and $z\in u(x)^\perp$,
we get
$$
\sqrt{1-\langle n_{\overline{K}}(w),u(x)\rangle ^2}^{-1}=
%\frac{\sqrt{1+\|\nabla f_{\widetilde{K}}(\bar w)\|^2}}{\|\nabla f_{\widetilde{K}}(\bar w)\|}=
\frac{\sqrt{1+(\|\bar w\|+o(\|\bar w\|))^2}}{\|\bar w\|+o(\|\bar w\|)}.
$$
Thus a simple application of the coarea formula yields that, for $t>0$ sufficiently small 
and $d\ge 2$, 
\begin{eqnarray*}
&&\int_{\partial K\cap H^+(x,t)}\varrho(w)\,\HH^{d-1}(dw)\\
&&=H_{d-1}(x)^{-1/2}\int_0^{t\langle x,u(x)\rangle}\int_{\partial\overline{K}\cap H(r)}
\overline{\varrho}(w)g(w)\sqrt{1-\langle n_{\overline{K}}(w),u(x)\rangle ^2}^{-1}\,\HH^{d-2}(dw)\, dr.
\end{eqnarray*}
Since also $\overline{K}$ has a rolling ball, the map $w\mapsto n_{\overline{K}}(w)$ 
is continuous, and therefore also
$$
r\mapsto \int_{\partial\overline{K}\cap H(r)}
\overline{\varrho}(w)g(w)\sqrt{1-\langle n_{\overline{K}}(w),
u(x)\rangle ^2}^{-1}\,\HH^{d-2}(dw)
$$
is continuous. This implies that
\begin{eqnarray*}
&&\frac{\,\partial }{\partial\, t}\int_{\partial K\cap H^+(x,t)}
\varrho(w)\,\HH^{d-1}(dw)\\
&&=\frac{\langle x,u(x)\rangle}{H_{d-1}(x)^{1/2}}\int_{\partial 
\overline{K}\cap H(t\langle x,u(x)\rangle)}
\overline{\varrho}(w)g(w)\sqrt{1-\langle n_{\overline{K}}(w),u(x)
\rangle ^2}^{-1}\,\HH^{d-2}(dw)\\
&&=\frac{ \langle x,u(x)\rangle}{H_{d-1}(x)^{1/2}} 
\int_{\partial\overline{K}\cap H(t\langle x,u(x)\rangle)}
\overline{\varrho}(w)g(w)\frac{\sqrt{1+(\sqrt{2t\langle x,u(x)\rangle}+o(\sqrt{t}))^2}}{\sqrt{2t\langle x,u(x)\rangle}+o(\sqrt{t})}
\,\HH^{d-2}(dw) .
\end{eqnarray*}
Clearly, we have $\overline\varrho(w)\to \overline\varrho(x)=\varrho(x)$ 
and $g(w)\to 1$, as $t\to 0^+$, uniformly with respect to 
$w\in \partial\overline K \cap H(t\langle x, u(x)\rangle)$.
Moreover, since
$$
\overline{\Gamma}:=\left\{x+z-\frac{1}{2}\|z\|^2u(x):z\in u(x)^\perp\right\}
$$
is the osculating paraboloid  of $\overline{K}$ and $\overline{\Gamma}$  
has rotational symmetry, we obtain for $s=s(t)$ that
\begin{eqnarray*}
%\label{dts}
\lim_{t\to 0^+}t^{-\frac{d-3}{2}}\cdot
\frac{\partial\, s}{\partial\, t}(t)
& =&\frac{\varrho(x)\langle x,u(x)\rangle}{H_{d-1}(x)^{1/2}}
\lim_{t\to 0^+}\left(
t^{-\frac{d-3}{2}}(d-1)\alpha_{d-1}
\frac{\sqrt{2t\langle x,u(x)\rangle}^{d-2}}{\sqrt{2t\langle x,u(x)\rangle}}\right)\\
&=&(d-1)\alpha_{d-1}H_{d-1}(x)^{-\frac{1}{2}}\varrho(x)\left(2\langle x,u(x)\rangle\right)^{\frac{d-3}{2}}
\langle x,u(x)\rangle\\
&=&
(d-1)\alpha_{d-1}
\varrho(x) 2^{\frac{d-3}2}\langle x,u(x)\rangle^{\frac{d-1}2}
H_{d-1}(x)^{-\frac12}.
\end{eqnarray*}
Thus we have shown that
\begin{equation}\label{dts}
\lim_{t\to 0^+}t^{-\frac{d-3}2}\cdot
\frac{\partial\, s}{\partial\, t}(t) =
(d-1)\cdot
\varrho(x)2^{\frac{d-3}2}\langle x,u(x)\rangle^{\frac{d-1}2}
H_{d-1}(x)^{-\frac12}\alpha_{d-1}.
\end{equation}
In the same way, we also obtain
\begin{equation}
\label{ts}
\lim_{t\to 0^+}t^{-\frac{d-1}2}\cdot s(t) =
\varrho(x)2^{\frac{d-1}2}\langle x,u(x)\rangle^{\frac{d-1}2}
H_{d-1}(x)^{-\frac12}\alpha_{d-1}.
\end{equation}
Observe that \eqref{dts} and \eqref{ts} are valid also for $d=2$. 
In particular, \eqref{dts} and \eqref{ts} imply that 
$J_\varrho(y,L)$ can be rewritten as (cf. \eqref{Jdef}) 
\begin{equation}
\label{Js}
J_\varrho(y,L)=(d-1)^{-1}G(x)^2
 \lim_{n\to\infty}\int_0^{g(y,n)}n^{\frac2{d-1}}
\PP_\varrho\left(\tilde{y}_s\not\in K_n|L \right)s^{-\frac{d-3}{d-1}}\,ds,
\end{equation}
where
$$
G(x):=(\alpha_{d-1})^{\frac{-1}{d-1}}\varrho(x)^{\frac{-1}{d-1}}H_{d-1}(x)^{\frac1{2(d-1)}}
$$
and
$$
\lim_{n\to\infty}n^{\frac12}g(y,n)=
\alpha_{d-1}\varrho(x)(2\langle u(x),x\rangle)^{\frac{d-1}2}H_{d-1}(x)^{-\frac{1}{2}}.
$$

Now we show that in the domain
of integration $g(y,n)$ can be replaced by $n^{-1/2}$, that is
\begin{equation}
\label{Js0}
J_\varrho(y,L)=(d-1)^{-1}G(x)^2
 \lim_{n\to\infty}\int_0^{n^{-1/2}}n^{\frac2{d-1}}
\PP_\varrho\left(\tilde{y}_s\not\in K_n|L \right)s^{-\frac{d-3}{d-1}}\,ds.
\end{equation}
 It follows from
Lemma~\ref{genup} and \eqref{ts} that there exist constants $c_0>0$
and $c_2>c_1>0$ depending
on $y$, $K$, $L$, $\varrho$ such that if $s>0$ is small enough, then
$$
\PP_\varrho\left(\tilde{y}_s\not\in K_n|L \right)\ll
(1-c_0s)^n,
$$
and if $n$ is large
and $s$ is
between $g(n,y)$ and $n^{-1/2}$, then $c_1n^{-1/2}<s<c_2n^{-1/2}$.
In particular, 
\begin{eqnarray*}
 && \lim_{n\to\infty}\int_{c_1n^{-1/2}}^{c_2n^{-1/2}}n^{\frac2{d-1}}
\PP_\varrho\left(\tilde{y}_s\not\in K_n|L \right)s^{-\frac{d-3}{d-1}}\,ds\\
&&\ll \lim_{n\to\infty}n^{\frac2{d-1}}\int_{c_1n^{-1/2}}^{c_2n^{-1/2}}
e^{-c_0ns}s^{-\frac{d-3}{d-1}}\,ds\\
&&\leq\lim_{n\to\infty}
c_2n^{\frac2{d-1}-\frac12}
e^{-c_1c_0n^{\frac12}}c_1^{-\frac{d-3}{d-1}}n^{\frac{d-3}{2(d-1)}}=0,
\end{eqnarray*}
and hence \eqref{Js} yields \eqref{Js0}.

Let $\pi:\R^d\to u(x)^\perp$ denote the orthogonal projection to $u(x)^\perp$. 
Using \eqref{x*x}, \eqref{nach4} and \eqref{ts}, we obtain
\begin{eqnarray}
\label{Exs*}
\lim_{s\to 0^+}s^{\frac{-1}{d-1}}\|\pi(x-\tilde{x}_s)\|&=&0,\\
%\label{Exs*height}
\lim_{s\to 0^+}s^{\frac{-2}{d-1}}\langle u(x),x-\tilde{x}_s\rangle&=&
\frac{1}{2}G(x)^2.\nonumber
\end{eqnarray}
Let $Q$ denote the second fundamental form of
$\partial K$ at $x$ (cf.~\eqref{diff}), 
considered as a function on $u(x)^\perp$. Then there are an 
orthonormal basis $v_1,\ldots,v_{d-1}$ of $u(x)^\perp$ and 
positive numbers $k_1,\ldots,k_{d-1}>0$ such that 
$$
Q\left(\sum_{i=1}^{d-1}z_iv_i\right)=\sum_{i=1}^{d-1}k_iz_i^2.
$$  
Further, let
$\pi$ be the orthogonal projection to $u(x)^\perp$, and define 
$$
E:=\{z\in u(x)^\perp:\, Q(z)\le 1\},
$$
which is the Dupin indicatrix of $K$ at $x$, whose half axes
are $k_i(x)^{-1/2}$, $i=1,\ldots,d-1$. In addition,
let $\Gamma$ be the convex hull of the osculating paraboloid of $K$ 
at $x\in\partial K$, that is
$$
\Gamma=\{x+z-tu(x):\, z\in u(x)^\perp, 
t\geq \mbox{$\frac12$}\,Q(z)\}.
$$
Hence, we have
$$
\Gamma\cap H(x,t)=x^*_t+\sqrt{2t\langle x,u(x)\rangle}\,E,
$$
and there exists an increasing function $\tilde{\mu}(s)$
with $\lim_{s\to 0^+}\tilde{\mu}(s)=1$ such that
\begin{equation}
\label{Exs}
\tilde{x}_s^*+\tilde{\mu}(s)^{-1}G(x)
\cdot s^{\frac1{d-1}}E\subset
K\cap\widetilde{H}(x,s)\subset
\tilde{x}_s^*+\tilde{\mu}(s)G(x)
\cdot s^{\frac1{d-1}}E,
\end{equation}
where $\tilde{x}_s^*:=x_t^*\in (x-\R_+ u(x))\cap \widetilde{H}(x,s)$, and $s$ and $t$ are related
by equation \eqref{sdef}. From \eqref{Exs*} it follows that also
\begin{equation}
\label{Exsb}
\tilde{x}_s+\tilde{\mu}(s)^{-1}G(x)
\cdot s^{\frac1{d-1}}E\subset
K\cap\widetilde{H}(x,s)\subset
\tilde{x}_s+\tilde{\mu}(s)G(x)
\cdot s^{\frac1{d-1}}E,
\end{equation}

The rest of the proof is devoted to identifying the asymptotic behavior of the
integral \eqref{Js0}.
First, we adjust the domain of integration and the integrand
in a suitable way. In a second step,
the resulting expression is compared to the case where $K$ is the unit ball.
We recall that $x_1,\ldots,x_n$ are random points in $\partial K$,
and we put $\Xi_n:=\{x_1,\ldots,x_n\}$, hence $K_n=[\Xi_n]$. For a finite 
set $X\subset\R^d$, let $\#X$ denote the cardinality of $X$.

\begin{lemma}
\label{pointsneary}
For $\varepsilon\in(0,1)$, there exist
$\alpha,\beta>1$ and an integer $k>d$, depending only on $\varepsilon$
and $d$, with the following property.
If $L\in {\mathcal L}_j^d$,
$y\in\partial (K|L)$, $x\in\partial K$ is a normal boundary point 
of $K$ such that $y=x|L$ and $H_{d-1}(x)>0$, 
and if $n>n_0$, where $n_0$ depends on $\varepsilon,x,K,\varrho, L$, 
then
$$
\int_0^{n^{-1/2}}
\PP_\varrho\left(\tilde{y}_s\not\in K_n|L \right)s^{-\frac{d-3}{d-1}}\,ds
=\int_{\frac{\varepsilon^{(d-1)/2}}n}^{\frac{\alpha}n}
\varphi(K,L,y,\varrho,\varepsilon,s)s^{-\frac{d-3}{d-1}}\,ds
+O\left(\frac{\varepsilon}{n^{\frac2{d-1}}}\right),
$$
where
$$
\varphi(K,L,y,\varrho,\varepsilon,s)=
\PP_\varrho\left(\left(\tilde{y}_s\not\in
([\widetilde{C}(x,\beta s)\cap\Xi_n]|L)\right)
\mbox{\rm and }
\left(\#(\widetilde{C}(x,\beta s)\cap\Xi_n)\leq k\right)\right).
$$
\end{lemma}

\begin{proof} Let $\varepsilon\in (0,1)$ be given. Then $\alpha>1$ is chosen such that
\begin{equation}
\label{alphadef}
2^{d-1+\frac{2d}{d-1}}
\int_{2^{-d}\alpha}^\infty e^{-r}r^{\frac{2}{d-1}-1}\, dr<\varepsilon.
\end{equation}
Further, we choose $\beta\ge (16^2(d-1))^{d-1}$ such that
\begin{equation}
\label{betadef}
2^{d-1}e^{-2^{-3d+2}\sqrt{\beta}\cdot \varepsilon^{\frac{d-1}2}}<
\varepsilon\cdot\alpha^{\frac{-2}{d-1}}  ,
\end{equation}
and then we fix an integer $k>d$ such that
\begin{equation}
\label{kdef}
\frac{(\alpha\beta)^k}{k!}<\frac{\varepsilon}{\alpha^{\frac2{d-1}} }.
\end{equation}
Lemma~\ref{pointsneary} follows from the following three statements,
which we will prove assuming that $n$ is sufficiently large.
\begin{enumerate}
\item[(i)] $\displaystyle
\int_0^{n^{-1/2}}\PP_\varrho\left(\tilde{y}_s\not\in K_n|L \right)
s^{-\frac{d-3}{d-1}}\,ds
=\int_{\frac{\varepsilon^{(d-1)/2}}n}^{\frac{\alpha}n}
\PP_\varrho\left(\tilde{y}_s\not\in K_n|L \right)s^{-\frac{d-3}{d-1}}\,ds
+O\left(\frac{\varepsilon}{n^{\frac2{d-1}}}\right).$
\item[(ii)] If ${\varepsilon^{(d-1)/2}}/n<s<{\alpha}/n$, then
$$
\PP_\varrho\left(\#\left(\widetilde{C}(x,\beta s)\cap\Xi_n\right)\geq k \right)
\leq\frac{\varepsilon}{\alpha^{\frac2{d-1}}}.
$$
\item[(iii)] If ${\varepsilon^{{(d-1)}/2}}/n<s<{\alpha}/n$, then
$$
\PP_\varrho\left(\tilde{y}_s\not\in K_n|L \right)=
\PP_\varrho\left(\tilde{y}_s\not\in
\left[(\widetilde{C}(x,\beta s)\cap \Xi_n)|L\right] \right)
+O\left(\frac{\varepsilon}{\alpha^{\frac2{d-1}}}\right).
$$
\end{enumerate}
Before proving (i), (ii) and (iii), we note that they imply
\begin{eqnarray*}
\int_0^{n^{-1/2}}
\PP_\varrho\left(\tilde{y}_s\not\in K_n|L \right)s^{-\frac{d-3}{d-1}}\,ds
&=&\int_{\frac{\varepsilon^{(d-1)/2}}n}^{\frac{\alpha}n}
\varphi(K,L,y,\varrho,\varepsilon,s)s^{-\frac{d-3}{d-1}}\,ds+\\
&& + \,O\left(\frac{\varepsilon}{\alpha^{\frac2{d-1}}}\right)
\int_{\frac{\varepsilon^{(d-1)/2}}n}^{\frac{\alpha}n}s^{-\frac{d-3}{d-1}}\,ds
+O\left(\frac{\varepsilon}{n^{\frac2{d-1}}}\right),
\end{eqnarray*}
which in turn  yields Lemma~\ref{pointsneary}.

First, we  introduce some notation. As before,
let $Q$ be the second fundamental form at $x\in\partial K$, and let
$v_1,\ldots,v_{d-1}$ be an orthonormal basis of $u(x)^\bot$ representing
the principal directions. In addition, let
$\Theta'_1,\ldots,\Theta'_{2^{d-1}}$ be the corresponding coordinate corners,
and for $i=1,\ldots,2^{d-1}$ and $s\in(0,n^{-1/2})$, let
$$
\widetilde{\Theta}_{i,s}=
\widetilde{C}(x,s)\cap\left(\tilde{x}_s+\left[\Theta'_i,\R_+x\right]
\right).
$$
Subsequently, we show that
\begin{equation}
\label{thetavol}
\lim_{s\to 0^+}s^{-1}\int_{\widetilde{\Theta}_{i,s}\cap\partial K}\varrho(z)\,
\HH^{d-1}(dz)=2^{-(d-1)}.
\end{equation}

In fact, since a ball rolls freely inside $K$, $\varrho$ is continuous and 
positive at $x$, and by \eqref{Exs*} we deduce that
\begin{eqnarray*}
&&\lim_{s\to 0^+} s^{-1}\int_{\widetilde{\Theta}_{i,s}\cap\partial K}\varrho(z)\,
\HH^{d-1}(dz)\\
&&=\varrho(x)\lim_{s\to 0^+} s^{-1}\HH^{d-1}\left(\widetilde{\Theta}_{i,s}\cap\partial K\right)\\
&&=\varrho(x)\lim_{s\to 0^+} s^{-1}\HH^{d-1}\left(\partial K\cap\widetilde{C}(x,s)\cap
\left(\tilde{x}_s^*+[\Theta'_i,\R_+u(x)]\right)\right).
\end{eqnarray*}
Let $\Psi:\partial \Gamma\cap C(x,{r}/{R})\to\partial K\cap C(x,{r}/{R})$
be the diffeomorphism which assigns to a point $z\in\partial \Gamma\cap \widetilde{H}(x,s)$ the unique
point $\Psi(z)\in\partial K\cap\left(\tilde{x}^*_s+\R_+(z-\tilde{x}^*_s)\right)$.
It follows from \eqref{Exs} that there exists an increasing function $\mu:\R_+\to\R_+$
with $\lim_{s\to 0^+}\mu(s)=1$ such that
$$
\mu(s)^{-1}\le {\rm Lip}(\psi\vert (\partial \Gamma\cap \widetilde{C}(x,s)))\le \mu(s).
$$
Thus we get
\begin{eqnarray*}
&&\lim_{s\to 0^+} s^{-1}\HH^{d-1}\left(\partial K\cap\widetilde{C}(x,s)\cap
\left(\tilde{x}_s^*+[\Theta'_i,\R_+u(x)]\right)\right)\\
&&=\lim_{s\to 0^+} s^{-1}\HH^{d-1}\left(\Psi\left(\partial \Gamma\cap\widetilde{C}(x,s)\cap
\left(\tilde{x}_s^*+[\Theta'_i,\R_+u(x)]\right)\right)\right)\\
&&=\lim_{s\to 0^+} s^{-1}\HH^{d-1}\left(\partial \Gamma\cap\widetilde{C}(x,s)\cap
\left(\tilde{x}_s^*+[\Theta'_i,\R_+u(x)]\right)\right)\\
&&=2^{-(d-1)}\lim_{s\to 0^+} s^{-1}\HH^{d-1}\left(\partial \Gamma\cap\widetilde{C}(x,s)\right).
\end{eqnarray*}
Now we can repeat the preceding argument in reverse order and finally use \eqref{sdef} to arrive at the assertion \eqref{thetavol}.

To prove (i), we observe that
$$
\int_0^{\frac{\varepsilon^{(d-1)/2}}n}
\PP_\varrho\left(\tilde{y}_s\not\in K_n|L \right)s^{-\frac{d-3}{d-1}}\,ds\leq
\int_0^{\frac{\varepsilon^{(d-1)/2}}n}s^{-\frac{d-3}{d-1}}\,ds\ll
\frac{\varepsilon}{n^{\frac2{d-1}}}.
$$
Let ${\alpha}/n<s<n^{-1/2}$, and let
$n$ be sufficiently large.
First, \eqref{except} yields that 
$$
\PP_\varrho\left(o\not\in K_n,\;\tilde{y}_s\not\in K_n|L \right)\leq
\varepsilon n^{-\frac{2}{d-1}}.
$$
On the other hand, if $o\in K_n$, then $\tilde{y}_s\not\in K_n|L$ implies that 
$\widetilde{\Theta}_{i,s}\cap K_n=\emptyset$ for
some $i\in \{1, \ldots , 2^{d-1}\}$, and hence
 (\ref{thetavol}) yields
\begin{equation}
\label{est}
\PP_\varrho\left(o\in K_n,\;\tilde{y}_s\not\in K_n|L \right)\leq
2^{d-1}(1-2^{-d}s)^n<2^{d-1}e^{-2^{-d}ns}.
\end{equation}
Therefore,  by \eqref{alphadef} we get
\begin{eqnarray*}
\int_{\alpha/n}^{n^{-1/2}}
\PP_\varrho\left(\tilde{y}_s\not\in K_n|L \right)s^{-\frac{d-3}{d-1}}\,ds
&\ll& 2^{d-1}\int_{\alpha/n}^\infty e^{-2^{-d}ns}s^{\frac{2}{d-1}-1}\, ds
+\frac{\varepsilon}{n^{\frac2{d-1}}}\\
&=& \frac{2^{d-1+\frac{2d}{d-1}}}{n^{\frac{2}{d-1}}}
\int_{2^{-d}\alpha}^\infty e^{-r}
r^{\frac{2}{d-1}-1}\, dr+\frac{\varepsilon}{n^{\frac2{d-1}}}\\
&\leq& \frac{2\varepsilon}{n^{\frac2{d-1}}},
\end{eqnarray*}
which verifies (i).

Next (ii) simply follows from \eqref{sdef} and
\eqref{kdef}. In fact, if $0<s< {\alpha}/n$, then
$$
\PP_\varrho\left(\#\left(\widetilde{C}(x,\beta s)\cap
\Xi_n\right)\geq k \right)\leq \binom{n}{k}(\beta s)^k
\le \binom{n}{k}\left(\frac{\alpha\beta}n\right)^k<
\frac{(\alpha\beta)^k}{k!}\le \frac{\varepsilon}{\alpha^{\frac2{d-1}}}.
$$

Finally, we prove (iii).
To this end, if ${\varepsilon^{(d-1)/2}}/n<s<{\alpha}/n$ 
and $i\in\{1,\ldots,2^{d-1}\}$, then
we define $w_i\in \Theta'_i$ by
\begin{equation}
\label{widef}
w_i:=\left(\sqrt{\beta}s\right)^{\frac{1}{d-1}}
\sum_{m=1}^{d-1}\frac{\eta_m G(x)}{4\sqrt{(d-1)k_m(x)}}\,v_m,
\end{equation}
where $\eta_m=\eta_m^i\in\{- 1,1\}$ for $m=1,\ldots,2^{d-1}$. Now let
$$
\widetilde{\Omega}_{i,s}:=
\partial K \cap
[\tilde{x}_{s}+\Theta'_i,
\tilde{x}_{\sqrt{\beta}\, s}+w_i+\Theta'_i].
$$
We claim that  for large $n$,
if $\tilde{y}_s\in K_{n}|L$
but $\tilde{y}_s\not\in\left[(\widetilde{C}(x,\beta s)\cap\Xi_n)|L\right]$,
then  there exists $i\in \{1,\ldots,2^{d-1}\}$ such that
\begin{equation}
\label{Omegaempty}
\Xi_n\cap\widetilde{\Omega}_{i,s}=\emptyset.
\end{equation}
Moreover, for all $i=1,\ldots,2^{d-1}$, we have
\begin{equation}
\label{Omegavol}
\int_{\widetilde{\Omega}_{i,s}}\varrho(z)\,\HH^{d-1}(dz)\geq
2^{-3d+2}\sqrt{\beta} s.
\end{equation}
To justify \eqref{Omegavol}, let $i\in\{1,\ldots,2^{d-1}\}$ be fixed.
It follows from the definition of $w_i$ that
$$
w_i\in\left(\sqrt{\beta}s\right)^{\frac{1}{d-1}}
\frac{G(x)}{4}\cdot \partial E.
$$
Recall that $\pi:\R^d\to u(x)^\perp$ denotes the orthogonal projection 
to $u(x)^\perp$. If $n$ is large enough,
and hence $0< s<{\alpha}/n$ is sufficiently small,
then \eqref{Exs*}, \eqref{Exsb} and \eqref{widef}
 yield that $w_i\in \pi(\widetilde{\Omega}_{i,s})$, since 
by assumption $\sqrt{\beta}^{1/(d-1)}/4>2$, and therefore
$$
(w_i+\Theta'_i)\cap \left(w_i+\left(\sqrt{\beta}s\right)^{\frac{1}{d-1}}
\frac{G(x)}{4}\cdot E\right)\subset\pi(\widetilde{\Omega}_{i,s}).
$$
In particular, \eqref{Omegavol} now follows from
\begin{eqnarray*}
\int_{\widetilde{\Omega}_{i,s}}\varrho(z)\,\HH^{d-1}(dz)
&\geq&\frac{\varrho(x)}2 \cdot \mathcal{H}^{d-1}(\widetilde{\Omega}_{i,s})\\
&\ge &\frac{\varrho(x)}2 \cdot \mathcal{H}^{d-1}(\pi(\widetilde{\Omega}_{i,s}))\\
&\geq&\frac{\varrho(x)}2 \cdot \frac{1}{2^{d-1}}\sqrt{\beta}s\frac{G(x)^{d-1}}{4^{d-1}}
\alpha_{d-1}H_{d-1}(x)^{-1/2}\\
&=&2^{-d}4^{1-d}\sqrt{\beta}s.
\end{eqnarray*}
Next we verify \eqref{Omegaempty}. 
We assume that $\tilde{y}_s\in K_{n}|L$ 
but $\tilde{y}_s\not\in \left[(\widetilde{C}(x,\beta s)\cap\Xi_n)|L\right]$. 
Then there exist 
$a\in \left[(\widetilde{C}(x,\beta s)\cap\Xi_n)|L\right]$
and $b\in \left(K_{n}\setminus \widetilde{C}(x,\beta s)\right)|L$ such that
$\tilde{y}_s\in (a,b)$. Thus
there exists a hyperplane $H$ in $\R^d$  containing $\tilde{y}_s+L^\bot$ and 
bounding the halfspaces $H^+$ and $H^-$ such that
$\widetilde{C}(x,\beta s)\cap\Xi_n\subset {\rm int}(H^+)$ and
$b\in {\rm int}(H^-)$. In addition, there exists 
$i\in \{1,\ldots,2^{d-1}\}$ such that
\begin{equation}
\label{corner}
\tilde{x}_s+\Theta'_i\subset H^-.
\end{equation}
Now we define points $q$ and $q'$ by 
$$
\{q\}=[\tilde{y}_s,b]\cap\widetilde{H}(x,\sqrt{\beta} s),\qquad 
\{q'\}=[\tilde{y}_s,b]\cap\widetilde{H}(x,{\beta} s).
$$
Relation \eqref{Exs} implies that
$$
\widetilde{H}(x,{\beta} s)\cap K\subset\tilde{x}^*_{\beta s}+
2 G(x)(\beta s)^{\frac{1}{d-1}} E
$$
if $s> 0$ is sufficiently small. Arguing as in \cite{BFH}, we obtain that
$$
\langle u(x),\tilde{y}_s-\tilde{y}_{\beta s}\rangle 
<\frac{\beta^{1/(d-1)}}{\beta^{1/(d-1)}-1}
\langle u(x),\tilde{y}_{\sqrt{\beta}s}-\tilde{y}_{\beta s}\rangle
$$
and
$$
\frac{\|q-\tilde{y}_{\sqrt{\beta}s}\|}{\|q'-\tilde{y}_{{\beta}s}\|}=
\frac{\langle u(x),\tilde{y}_s-\tilde{y}_{\sqrt{\beta} s}\rangle}{
\langle u(x),\tilde{y}_s-\tilde{y}_{\beta s}\rangle},
$$
which yields (cf.\ \cite{BFH})
$$
q\in \tilde{y}_{\sqrt{\beta} s}+2s^{\frac{1}{d-1}}G(x) E.
$$
Since $\beta\ge [8^2(d-1)]^{d-1}$, we thus arrive at
\begin{equation}\label{newstar}
q\in \tilde{y}_{\sqrt{\beta} s}+\frac{1}{4\sqrt{d-1}}
(\sqrt{\beta}s)^{\frac{1}{d-1}}G(x)E.
\end{equation}
Now \eqref{corner} implies that $q+\Theta_i'\subset H^-$. 
Hence it follows from \eqref{newstar} that $\tilde{y}_{\sqrt{\beta} s}
+w_i\subset q+\Theta_i'\subset H^-$, and therefore 
also $\tilde{y}_{\sqrt{\beta} s}+w_i+\Theta_i'\subset H^-$. Thus 
$\widetilde{\Omega}_{i,s}\subset H^-$, which yields 
$\Xi_n\cap \widetilde{\Omega}_{i,s}=\emptyset$. 

Assertion (iii) follows from \eqref{Omegaempty} and \eqref{Omegavol}. In fact, 
if ${\varepsilon^{{(d-1)}/2}}/n<s<{\alpha}/n$, then
\begin{eqnarray*}
&&\PP_\varrho\left(\tilde{y}_s\not\in\left[(\widetilde{C}(y,\beta s)\cap
\Xi_n)|L\right] \right)
-\PP_\varrho\left(\tilde{y}_s\not\in (K_{n}|L) \right)\\
&&\leq \sum_{i=1}^{2^{d-1}}
\left(1-\int_{\widetilde{\Omega}_{i,s}}\varrho(z)\,\HH^{d-1}(dz) \right)^n\\
&&\leq 2^{d-1}e^{-2^{-3d+2}\sqrt{\beta}\cdot sn}\\
&&\leq {\varepsilon}\,{\alpha^{-\frac2{d+1}} },
\end{eqnarray*}
by the choice of $\beta$.
\end{proof}

To actually compare the situation near the normal 
boundary point $x$ of $K$ with $H_{d-1}(x)>0$ to
the  case of the unit ball,
let $\sigma=(d\alpha_d)^{-1}$ be the constant density 
of the corresponding probability distribution on $S^{d-1}$. 
Let $w\in S^{d-1}$ be the $d$-th coordinate vector in $\R^d$,
and hence $\R^{d-1}=w^\bot$.
We write $B_n$ to
denote the convex hull of $n$ random points
distributed uniformly and independently
 on $S^{d-1}$ according to $\sigma$.
For $s\in (0,\frac12)$, we fix a linear subspace 
$L_0\in {\mathcal L}_j^d$ with $w\in L_0$, and
let $\tilde{w}_s$ be of the form $\lambda w$ for $\lambda\in(0,1)$ such that
$$
(d\alpha_d)^{-1}\cdot \HH^{d-1}(\{z\in S^{d-1}:\,\langle z,w\rangle\geq
\langle \tilde{w}_s,w\rangle\})=s.
$$
In particular,  $\tilde{w}_s|L_0=\tilde{w}_s$.

\begin{lemma}
\label{compareball}
If $L\in {\mathcal L}_j^d$,
$y\in\partial (K|L)$ and $x\in\partial K$ is a normal boundary 
point such that $y=x|L$ and $H_{d-1}(x)>0$, then
$$
\lim_{n\to\infty}\int_0^{n^{-1/2}}n^{\frac2{d-1}}
\PP_\varrho\left(\tilde{y}_s\not\in K_n|L \right)s^{-\frac{d-3}{d-1}}\,ds
=\lim_{n\to\infty}\int_0^{n^{-1/2}}n^{\frac2{d-1}}
\PP_\sigma\left(\tilde{w}_s\not\in B_n|L_0 \right)s^{-\frac{d-3}{d-1}}\,ds.
$$
\end{lemma}

\begin{proof} First, we assume $d\geq 3$.
It is sufficient to prove
that for any $\varepsilon\in(0,1)$ 
there exists $n_0>0$, depending on $\varepsilon,x,K,\varrho, L$, 
such that if $n>n_0$, then
\begin{equation}
\label{Kballeps}
\int_0^{n^{-1/2}}
\PP_\varrho\left(\tilde{y}_s\not\in K_n|L \right)s^{-\frac{d-3}{d-1}}\,ds
=\int_0^{n^{-1/2}}
\PP_\sigma\left(\tilde{w}_s\not\in B_n|L_0 \right)s^{-\frac{d-3}{d-1}}\,ds
+O\left(\frac{\varepsilon}{n^{\frac2{d-1}}}\right).
\end{equation}
 Let $\alpha$, $\beta$ and
$k$ be the quantities associated with $\varepsilon,x,K,\varrho, L$ 
in Lemma~\ref{pointsneary}, let
$\widetilde{C}(x,s)$ denote the cap of $K$ defined in
\eqref{capK}, and let $\widetilde{C}(w,s)$ denote the
corresponding cap of $B^d$ at $w$.
We define the densities  $\varrho_s$
on $\partial\widetilde{C}(x,\beta s)$ and $\sigma_s$
on $\partial\widetilde{C}(w,\beta s)$ of probability distributions 
 by
\begin{eqnarray*}
\varrho_s(z)&=&\left\{
\begin{array}{rl}
\varrho(z)/(\beta s),&
\mbox{ \ if $z\in\partial K\cap \widetilde{C}(x,\beta s)$},\\
0,&\mbox{ \ if $z\in\partial \widetilde{C}(x,\beta s)\backslash\partial K$},\\
\end{array}\right.\\
\sigma_s(z)&=&\left\{
\begin{array}{rl}
\sigma(z)/(\beta s),&
\mbox{ \ if $z\in S^{d-1}\cap \widetilde{C}(w,\beta s)$},\\
0,&\mbox{ \ if $z\in\partial \widetilde{C}(w,\beta s)\backslash S^{d-1}$}.\\
\end{array}\right. 
\end{eqnarray*}
For $i=0,\ldots,k$,
we write $\widetilde{C}(x,\beta s)_i$ and  $\widetilde{C}(w,\beta s)_i$ to
denote the convex hulls of $i$ random points
distributed uniformly and independently
 on $\partial \widetilde{C}(x,\beta s)$
and $\partial \widetilde{C}(w,\beta s)$ according to
$\varrho_s$ and
$\sigma_s$, respectively.

If $n$ is large, then  Lemma~\ref{pointsneary} yields
that the left-hand and the right-hand side of \eqref{Kballeps}
are
$$
O\left(\frac{\varepsilon}{n^{\frac2{d-1}}}\right)+
\sum_{i=0}^k
\binom{n}{i}\int_{\frac{\varepsilon^{{(d-1)}/2}}n}^{\frac{\alpha}n}
(\beta s)^i(1-\beta s)^{n-i}
\times\,
\PP_{\varrho_{s}}
\left(\tilde{y}_s\not\in \widetilde{C}(x,\beta s)_i|L \right)
s^{-\frac{d-3}{d-1}}\,ds,
$$
$$
O\left(\frac{\varepsilon}{n^{\frac2{d-1}}}\right)+
\sum_{i=0}^k
\binom{n}{i}\int_{\frac{\varepsilon^{{(d-1)}/2}}n}^{\frac{\alpha}n}
(\beta s)^i(1-\beta s)^{n-i}
\times\,
\PP_{\sigma_{s}}
\left(\tilde{w}_s\not\in \widetilde{C}(w,\beta s)_i|L_0 \right)
s^{-\frac{d-3}{d-1}}\,ds.
$$
For each $i\leq k$, the representation of 
the beta function by the gamma function and the Stirling formula
(see E. Artin \cite{Art64}) imply 
\begin{equation}
\label{Artin}
\lim_{n\to\infty}n^{\frac2{d-1}}\binom{n}{i}\int_{0}^{1/\beta}
(\beta s)^i(1-\beta s)^{n-i}s^{-\frac{d-3}{d-1}}\,ds
=\frac{\beta^{\frac{-2}{d-1}}\Gamma\left(i+\frac2{d-1}\right)}{i!}<1.
\end{equation}
Therefore to prove \eqref{Kballeps}, it is sufficient
to verify that for each $i=0,\ldots,k$, if $s>0$
is small, then
\begin{equation}
\label{Kballeps0}
\left|
\PP_{\varrho_{s}}
\left(\tilde{y}_s\not\in \widetilde{C}(x,\beta s)_i|L \right)-
\PP_{\sigma_{s}}
\left(\tilde{w}_s\not\in \widetilde{C}(w,\beta s)_i|L_0 \right)
\right| \ll \frac{\varepsilon}{k}.
\end{equation}
If $i\leq j$, then \eqref{Kballeps0} readily holds
as its left-hand side is zero.

 To prove \eqref{Kballeps0} if $i\in\{j+1,\ldots,k\}$,
we transform both $K$ and $B^d$ in such a way that their osculating paraboloid is
$\Omega=\{z-\|z\|^2\,w:\,z\in\R^{d-1}\}$, and the images
of the caps $\widetilde{C}(x,\beta s)$ and $\widetilde{C}(w,\beta s)$
are very close. Using these caps, we construct equivalent
representations of $\PP_{\varrho_{s}}
\left(\tilde{y}_s\not\in \widetilde{C}(x,\beta s)_i|L \right)$
and $\PP_{\sigma_{s}}
\left(\tilde{w}_s\not\in \widetilde{C}(w,\beta s)_i|L_0 \right)$, 
based on the same space $\Xi_s$ and on 
 comparable probability measures and random variables.

We may assume that $u(x)=w$.
Let $v_1,\ldots,v_{d-1}$ be an orthonormal basis
of $w^\bot$ in the principal directions
of the fundamental form $Q$ of $K$ at $x\in\partial K$. We define
the linear transform $A_s$ of $\R^d$ by
\begin{eqnarray*}
A_s(w)&=&2(\beta s)^{\frac{-2}{d-1}}G(x)^{-2}w,\\
A_s(v_i)&=&(\beta s)^{\frac{-1}{d-1}}\sqrt{k_i(x)}G(x)^{-1}v_i,\;\;
i=1,\ldots,d-1,
\end{eqnarray*}
and choose an orthonormal linear transform
$P_s$ such that $P_sw=w$, and $P_s\circ A_s(L^\bot)=L_0^\bot$.
Based on these linear transforms, let $\Phi_s$ be
the affine transformation
$$
\Phi_s(z)=P_s\circ A_s(z-x).
$$
In addition, we define the linear transform $R_s$ of $\R^d$ by
\begin{eqnarray*}
R_s(w)&=&2(\beta s)^{\frac{-2}{d-1}}
\left(\frac{\alpha_{d-1}}{d\alpha_d}\right)^{\frac{2}{d-1}}w,\\
R_s(v_i)&=&(\beta s)^{\frac{-1}{d-1}}
\left(\frac{\alpha_{d-1}}{d\alpha_d}\right)^{\frac{1}{d-1}}v_i,\;\;
i=1,\ldots,d-1,
\end{eqnarray*}
and  let $\Psi_s$ be
the affine transformation
$$
\Psi_s(z)=R_s(z-x).
$$
Subsequently, we also write $\Phi_s z$ for $\Phi_s(z)$ or $\Phi_s z|L_0$ for 
$\Phi_s(z)|L_0$, and similarly for $\Psi_s$. 
We observe that $\Omega$ is
the osculating paraboloid of both $\Phi_sK$ and
$\Psi_s B^d$ at $o$, and
\begin{eqnarray*}
\lim_{s\to 0^+}\Phi_s\tilde{x}_s=\lim_{s\to 0^+}\Psi_s\tilde{w}_s&=&
-\beta^{\frac{-2}{d-1}}w=:w^*,\\
\lim_{s\to 0^+}\Phi_s\widetilde{C}(x,\beta s)=
\lim_{s\to 0^+}\Psi_s\widetilde{C}(w,\beta s)&=&
\{z-\tau\,w:\,z\in B^{d-1}\mbox{ and }\|z\|^2\leq \tau\leq 1\}.
\end{eqnarray*}
For $p\in \widetilde{C}(x,\beta s)\cap\partial K$
and $z=\pi\circ\Phi_s(p)$, let $D(p)$ be the Jacobian
of $\pi\circ\Phi_s$ at $p$ as a map
$\pi\circ\Phi_s:\widetilde{C}(x,\beta s)\cap\partial K\to\R^{d-1}$,
and let
$$
\tilde{\varrho}_s(z)=\varrho_s(p)\cdot D(p)^{-1}.
$$
In addition, for $p\in \widetilde{C}(w,\beta s)\cap S^{d-1}$
and $z=\pi\circ\Psi_s(p)$, let $\widetilde{D}(p)$ be the Jacobian
of $\pi\circ\Psi_s$ at $p$ as a map
$\pi\circ\Psi_s:\widetilde{C}(w,\beta s)\cap S^{d-1}\to\R^{d-1}$,
and let
$$
\tilde{\sigma}_s(z)=\sigma_s(p)\cdot \widetilde{D}(p)^{-1}.
$$
We define
$$
\Xi_s=\left[\pi\circ\Phi_s\widetilde{C}(x,\beta s)\right]
\cup \left[\pi\circ\Psi_s\widetilde{C}(w,\beta s)\right],
$$
and extend $\tilde{\varrho}_s$ and $\tilde{\sigma}_s$
to $\Xi_s$ by
\begin{eqnarray*}
 \tilde{\varrho}_s(z)&=&0,\mbox{ \ if \ }
z\in \left[\pi\circ\Psi_s\widetilde{C}(w,\beta s)\right]
\backslash \left[\pi\circ\Phi_s\widetilde{C}(x,\beta s)\right],\\
\tilde{\sigma}_s(z)&=&0,\mbox{ \ if \ }
\left[\pi\circ\Phi_s\widetilde{C}(x,\beta s)\right]
\backslash \left[\pi\circ\Psi_s\widetilde{C}(w,\beta s)\right].
\end{eqnarray*}
Therefore $\tilde{\varrho}_s$ and $\tilde{\sigma}_s$
are densities of probability distributions on $\Xi_s$.
For $z\in\Xi_s$, let $\varphi_s(z)\in \Phi_s\partial K$
and  $\psi_s(z)\in \Psi_s S^{d-1}$ be the points near $z$ whose
orthogonal projection into $\R^{d-1}$ is $z$.
For random variables $z_1,\ldots,z_i\in\Xi_s$
either with respect to
$\tilde{\varrho}_s$ or $\tilde{\sigma}_s$, the
quantities above were defined so as to satisfy
\begin{eqnarray}
\label{KXi}
 \PP_{\varrho_{s}}
\left(\tilde{y}_s\not\in \widetilde{C}(x,\beta s)_i|L \right)&=&
\PP_{\tilde{\varrho}_{s}}
\left(\Phi_s\tilde{x}_s|L_0\not\in
[\varphi_s(z_1),\ldots,\varphi_s(z_i)]|L_0 \right) ,\\
\label{BXi}
 \PP_{\sigma_{s}}
\left(\tilde{w}_s\not\in \widetilde{C}(w,\beta s)_i|L \right)&=&
\PP_{\tilde{\sigma}_{s}}
\left(\Psi_s\tilde{w}_s\not\in
[\psi_s(z_1),\ldots,\psi_s(z_i)]|L_0 \right).
\end{eqnarray}
Now there exists an increasing function $s\mapsto \mu^*(s)$
with $\lim_{s\to 0^+}\mu^*(s)=1$ such that
$$
\mu^*(s)^{-1}B^{d-1}\subset
\left[\pi\circ \Phi_s\widetilde{C}(x,\beta s)\right]
\cap \left[\pi\circ\Psi_s\widetilde{C}(w,\beta s)\right]
\subset \Xi_s \subset
 \mu^*(s)B^{d-1},
$$
we have $\mu^*(s)^{-1}\varphi_s(z)\leq\psi_s(z)\leq
\mu^*(s)\varphi_s(z)$ for all $z\in\Xi_s$,
and
\begin{eqnarray*}
\mu^*(s)^{-1}\alpha_{d-1}^{-1}&\leq \tilde{\varrho}_s(z)\leq &
\mu^*(s)\alpha_{d-1}^{-1}, \quad\mbox{ \ if } 
z\in\pi\circ\Phi_s\widetilde{C}(x,\beta s),\\
\mu^*(s)^{-1}\alpha_{d-1}^{-1}&\leq \tilde{\sigma}_s(z)\leq &
\mu^*(s)\alpha_{d-1}^{-1}, \quad\mbox{ \ if } 
z\in\pi\circ\Psi_s\widetilde{C}(w,\beta s).
\end{eqnarray*}
Therefore
\begin{equation}
\label{rhosigma}
\lim_{s\to 0^+}\int_{\Xi_s}|\tilde{\varrho}_s(z)- \tilde{\sigma}_s(z)|\,\HH^{d-1}(dz)=0.
\end{equation}
From \eqref{rhosigma} we deduce that if $s>0$ is small, then
\begin{eqnarray}
\label{noKB}
\left|\PP_{\tilde{\varrho}_{s}}
\left(\Phi_s\tilde{x}_s|L_0\not\in
[\varphi_s(z_1),\ldots,\varphi_s(z_i)]|L_0 \mbox{ and }
\Psi_s\tilde{w}_s\not\in
[\psi_s(z_1),\ldots,\psi_s(z_i)]|L_0\right)-\right.&& \\
\nonumber
\left.\PP_{\tilde{\sigma}_{s}}
\left(\Phi_s\tilde{x}_s|L_0\not\in
[\varphi_s(z_1),\ldots,\varphi_s(z_i)]|L_0 \mbox{ and }
\Psi_s\tilde{w}_s\not\in
[\psi_s(z_1),\ldots,\psi_s(z_i)]|L_0\right)  \right|
&\leq&\frac{\varepsilon}{k}.
\end{eqnarray}
Next, if $s>0$ is small, then
$$
\left\|w^*-\Phi_s\tilde{x}_s\right\|
\leq\frac{\varepsilon}{k^{j+1}}\quad\mbox{ \ and \ }\quad 
\left\|w^*-\Psi_s\tilde{w}_s\right\|
\leq\frac{\varepsilon}{k^{j+1}},
$$
and in addition
$$
\|\varphi_s(z)-\psi_s(z)\|\leq \frac{\varepsilon}{ k^{j+1}}
\quad \mbox{ \ \ for all }z\in\Xi_s.
$$
Let us assume that $\Phi_s\tilde{x}_s|L_0\not\in
[\varphi_s(z_1),\ldots,\varphi_s(z_i)]|L_0$
but $\Psi_s\tilde{w}_s\in [\psi_s(z_1),\ldots,\psi_s(z_i)]|L_0$
for some $z_1,\ldots,z_i\in\Xi_s$. In this case, the point $a$ of 
 $[\varphi_s(z_1),\ldots,\varphi_s(z_i)]|L_0$
closest to $\Phi_s\tilde{x}_s|L_0$ is
contained in some $(j-1)$-simplex
$[\varphi_s(z_{m_1}),\ldots,\varphi_s(z_{m_j})]|L_0$, i.e.\ there are 
$\lambda_1,\ldots,\lambda_j\ge 0$, $\lambda_1+\ldots+\lambda_j=1$, such that 
$a=\sum_{r=1}^j \lambda_r\varphi(z_{m_r})|L_0$. Moreover, there are 
$\mu_1,\ldots,\mu_i\ge 0$, $\mu_1+\ldots+\mu_i=1$, so that $\Psi_s\tilde{w}_s=
\sum_{r=1}^i\mu_r\psi_s(z_r)|L_0$. Then we have
\begin{eqnarray*}
\|\Phi_s \tilde{x}_s|L_0-a\|&\le & \left\|\Phi_s \tilde{x}_s|L_0-\sum_{r=1}^i\mu_r\varphi_s(z_r)|L_0\right\|\\
&\le &\|\Phi_s \tilde{x}_s|L_0-w^*\|+\|w^*-\Psi_s \tilde{w}_s\|+\left\|\sum_{r=1}^i\mu_r(\psi_s(z_r)-\varphi_s(z_r))|L_0\right\|\\
&\le&\frac{\varepsilon}{k^{j+1}}+\frac{\varepsilon}{k^{j+1}}+\frac{\varepsilon}{k^{j+1}}=\frac{3\varepsilon}{k^{j+1}},
\end{eqnarray*}
and hence
$$
\|w^*-a\|\le \frac{4\epsilon}{k^{j+1}}.
$$
Choose a maximal set $v_1,\ldots,v_l\in S^{d-1}\cap L_0$ such that
the distance between any two points is at least
$\varepsilon k^{-(j+1)}$, in particular
$$
l\ll \varepsilon^{-(j-1)} k^{(j-1)(j+1)}.
$$
Since $a,\varphi_s(z_{m_1})|L_0,\ldots,\varphi_s(z_{m_j})|L_0$ 
lie in a $(j-1)$-dimensional affine subspace of $L_0$, there is a unit vector 
$v\in S^{d-1}\cap L_0$ such that $|\langle \varphi_s(z_{m_r})-w^*,v\rangle|\le  4\varepsilon k^{-(j+1)}$ for $r=1,\ldots,j$, and thus 
$$|\langle \varphi_s(z_{m_r})-w^*,v_m\rangle|\le \frac{6\varepsilon}{k^{j+1}}$$  
for $r=1,\ldots,j$ and a 
suitably chosen $m\in \{1,\ldots,l\}$. In fact, for the given vector $v\in S^{d-1}\cap L_0$, there 
is some $m\in\{1,\ldots,l\}$ such that $\|v-v_m\|\le \varepsilon k^{-(j+1)}$. Since 
$\Phi_s\widetilde{C}(x,\beta s)\subset w^*+2B^d$, we deduce that 
\begin{eqnarray*}
|\langle \varphi_s(z_{m_r})-w^*,v_m\rangle|&\le &|\langle \varphi_s(z_{m_r})-w^*,v\rangle|
+\|\varphi_s(z_{m_r})-w^*\|\cdot\|v_m-v \|\\
&\le &\frac{4\epsilon}{k^{j+1}}+2\cdot \frac{\epsilon}{k^{j+1}}=\frac{6\epsilon}{k^{j+1}}.
\end{eqnarray*}
Therefore, if we define, for $m=1,\ldots,l$,
$$
\Pi_m:=\left\{p\in
\partial\Phi_s\widetilde{C}(x,\beta s):\,
|\langle p-w^*,v_m\rangle|\leq 6\varepsilon k^{-(j+1)}\right\},
$$
we get the following: if $\Phi_s\tilde{x}_s|L_0\not\in
[\varphi_s(z_1),\ldots,\varphi_s(z_i)]|L_0$
but $\Psi_s\tilde{w}_s\in [\psi_s(z_1),\ldots,\psi_s(z_i)]|L_0$
for some $z_1,\ldots,z_i\in\Xi_s$,  then there exists $m\in\{1,\ldots,l\}$ 
such that $\Pi_m$ contains some $j$ of the points
$\varphi_s(z_1),\ldots,\varphi_s(z_i)$.
Since $\HH^{d-1}(\Pi_m)\ll \varepsilon k^{-(j+1)}$, we have
\begin{eqnarray}
\nonumber
&&\PP_{\tilde{\varrho}_{s}}
\left(\Phi_s\tilde{x}_s|L_0\not\in
[\varphi_s(z_1),\ldots,\varphi_s(z_i)]|L_0 \mbox{ and }
\Psi_s\tilde{w}_s\in
[\psi_s(z_1),\ldots,\psi_s(z_i)]|L_0\right)\\
&&\leq 
\binom{i }{ j}\sum_{m=1}^l\PP_{\tilde{\varrho}_{s}}
\left(\varphi_s(z_1),\ldots,\varphi_s(z_j)\in\Pi_m\right)\nonumber\\
&&\ll \binom{i }{ j}\cdot l\cdot (\varepsilon k^{-(j+1)})^j \ll 
\frac{\varepsilon}{k}.\label{KnoB}
\end{eqnarray}
Similarly, we have
\begin{equation}
\label{BnoK}
\PP_{\tilde{\sigma}_{s}}
\left(\Psi_s\tilde{w}_s\not\in
[\psi_s(z_1),\ldots,\psi_s(z_i)]|L_0 \mbox{ and }
\Phi_s\tilde{x}_s|L_0\in
[\varphi_s(z_1),\ldots,\varphi_s(z_i)]|L_0\right)\ll
\frac{\varepsilon}{k}.
\end{equation}
Combining \eqref{KXi}, \eqref{BXi} as well as \eqref{noKB}, \eqref{KnoB} and \eqref{BnoK}
yields \eqref{Kballeps0}, and in turn Lemma~\ref{compareball}
if $d\geq 3$.

If $d=2$, then a similar argument works, only some of the constrains
should be modified as follows.
In \eqref{Artin}, we only have
${\beta^{\frac{-2}{d-1}}\Gamma\left(i+\frac2{d-1}\right)}/{i!}<k+1$,
and hence in \eqref{Kballeps0}, we should verify
an upper bound  of order $\frac{\varepsilon}{k^2}$, not
of order $\frac{\varepsilon}{k}$. Therefore
the upper bound in \eqref{noKB} should be $\frac{\varepsilon}{k^2}$.
\end{proof}

\section{Completing the proof of Theorem~\ref{main}}

In order to transfer an integral over an average of
projections of a convex body
to a boundary integral, we are going to use the following
lemma from
K. B\"or\"oczky Jr., L. M. Hoffmann, D. Hug \cite{BoHoHu2008}. 

For $L\in {\mathcal L}_j^d$ and $y\in\partial (K|L)$, we choose a point 
$x(y)\in \partial K$ such that $y=x(y)|L$. In general, $x(y)$ is not 
uniquely determined, but we can fix a measurable choice 
(cf.\ \cite[p.\ 152]{BoHoHu2008}). Recall, however, that $x(y)$ is 
uniquely determined for $\nu_j$ a.e.\ $L\in {\mathcal L}_j^d$ and 
$\mathcal{H}^{j-1}$ a.e.\ $y\in\partial (K|L)$. 

\begin{lemma}
\label{curvature}
Let $K\subset\R^d$ be a convex body in which a ball rolls freely,
 let $f:\partial K\to [0,\infty)$ be nonnegative and measurable,
 and let $j\in\{1,\ldots,d-1\}$. Then
$$
\frac{j\alpha_j}{d\alpha_d}\int_{\partial K}f(x)H_{d-j}(x)\,
\mathcal{H}^{d-1}(dx)=
\int_{\mathcal{L}^d_j}\int_{\partial(K|L)}f(x(y))\,
\mathcal{H}^{j-1}(dy)\,\nu_j(dL).
$$
\end{lemma}

By the very special case $K=B^d$ of \eqref{reitzner}, 
due to M.\ Reitzner  \cite{Reitz2002}, we have
$$
\lim_{n\to\infty}n^{\frac{2}{d-1}}\left[V_j(B^d)-\E_\sigma V_j(B_n)\right]=
c^{(j, d)}(d\alpha_d)^{\frac{d+1}{d-1}}.
$$
Therefore the rotational symmetry of $B^d$,
\eqref{limitform} and \eqref{Js0} yield
\begin{eqnarray}
\label{BallJ}
c^{(j, d)}(d\alpha_d)^{\frac{d+1}{d-1}}\nonumber&=&
\frac{\binom{d}{ j}\alpha_d}{\alpha_{d-j}\alpha_j}
\cdot\frac{j\alpha_j(d\alpha_d)^{\frac{2}{d-1}}}{d-1}(\alpha_{d-1})^{-\frac{2}{d-1}}\\
&&\qquad\times 
\lim_{n\to\infty}\int_0^{n^{-1/2}}n^{\frac2{d-1}}
\PP_\sigma\left(\tilde{w}_s\not\in B_n|L_0 \right)s^{-\frac{d-3}{d-1}}\,ds.
\end{eqnarray}

We can now transform the asymptotic formulas to $K$.
Let $L\in {\mathcal L}_j^d$ and let $y\in\partial (K|L)$ be 
such that $y=x|L$ for some normal boundary point $x=x(y)\in \partial K$.
If $H_{d-1}(x)=0$, then  $J_\varrho(y, L)=0$ by Lemma~\ref{zerocurv}.
If $H_{d-1}(x)>0$, then
it follows from  \eqref{Js0}, Lemma~\ref{compareball} and \eqref{BallJ} that
\begin{eqnarray*}
 J_\varrho(y,L)&=& (d-1)^{-1}(\alpha_{d-1})^{-\frac{2}{d-1}} 
\varrho(x)^{\frac{-2}{d-1}}H_{d-1}(x)^{\frac{1}{d-1}}\\
&&\qquad \times \lim_{n\to\infty}\int_0^{n^{-1/2}}n^{\frac2{d-1}}
\PP_\sigma\left(\tilde{w}_s\not\in B_n|L_0 \right)s^{-\frac{d-3}{d-1}}\,ds\\
&=&c^{(j, d)}\varrho(x)^{\frac{-2}{d-1}}H_{d-1}(x)^{\frac{1}{d-1}}\left(
\frac{\binom{d}{ j}\alpha_d}{\alpha_{d-j}\alpha_j}\cdot
\frac{j\alpha_j}{d\alpha_d}\right)^{-1},
\end{eqnarray*}
where $x=x(y)$. 
Finally, we apply first \eqref{limitform}, and afterwards Lemma~\ref{curvature}, 
to deduce
\begin{eqnarray*}
&& \lim_{n\to\infty}n^{\frac{2}{d-1}}\left[V_j(K)-\E_\varrho(V_j(K_n))\right]\\
&&=
c^{(j, d)}\,\frac{d\alpha_d}{j\alpha_j}
\int_{{\mathcal L}_j^d}
\int_{\partial (K| L)} \varrho(x(y))^{\frac{-2}{d-1}}H_{d-1}(x(y))^{\frac{1}{d-1}}\, 
\HH^{j-1}(dy)\,\nu_j(dL)\\
&&= c^{(j, d)}
\int_{\partial K}\varrho(x)^{\frac{-2}{d-1}}H_{d-1}(x)^{\frac{1}{d-1}}\, 
H_{d-j}(x)\,\mathcal{H}^{d-1}(dx),
\end{eqnarray*}
which concludes the proof of Theorem~\ref{main}.

\section{Proof of Theorem~\ref{meanwidthbounds}}

Using the Stirling formula $\Gamma(n+1)\sim(\frac{n}e)^n\sqrt{2\pi n}$, 
as $n\to\infty$ (see E. Artin \cite{Art64}),
for any $\alpha>0$ and $\gamma\in(0,1]$, we deduce
\begin{eqnarray}
\label{Gammalimit}
\lim_{n\to\infty}n^{\alpha}\int_0^\gamma s^{\alpha-1} (1-s)^n\,ds
&=&
\lim_{n\to\infty}n^{\alpha}\int_0^1 s^{\alpha-1} (1-s)^n\,ds\nonumber\\
&=&\lim_{n\to\infty}n^{\alpha}\frac{\Gamma(\alpha)\Gamma(n+1)}{\Gamma(n+1+\alpha)}=
\Gamma(\alpha).
\end{eqnarray}
In the following argument, $\gamma_1,\gamma_2,\ldots$ again 
denote positive constants that may depend on $K$ and $\varrho$. 
We can assume that $o\in \text{int}(K)$. Further, let 
$(\partial K)^n_*$ denote the set of all $x_1,\ldots,x_n\in\partial K$ 
such that $o\in[x_1,\ldots,x_n]$. 
For $u\in S^{d-1}$ and $t\geq 0$, let
$$
C(u,t):=\{x\in K:\,\langle x,u\rangle\geq h_K(u)-t\},
$$
where $h_K$ denotes the support function of $K$. 
To deduce the upper bound, we start with the estimates
\begin{eqnarray}
&&\E_\varrho (V_1(K)-V_1(K_{n}))\nonumber\\
&&=\frac{1}{\alpha_{d-1}}\int_{(\partial K)^n}\int_{S^{d-1}} (h_K(u)-h_{K_n}(u))\, 
\HH^{d-1}(du)\varrho(x_1)\cdots\varrho(x_{n})\, 
\HH^{d-1}(dx_1)\ldots \HH^{d-1}(dx_n)\nonumber\\
&&\le\frac{1}{\alpha_{d-1}}\int_{(\partial K)^n_*}\int_{S^{d-1}} (h_K(u)-h_{K_n}(u))\, 
\HH^{d-1}(du)\varrho(x_1)\cdots\varrho(x_{n})\, 
\HH^{d-1}(dx_1)\ldots \HH^{d-1}(dx_n)\nonumber\\
&&\qquad\qquad +\,2^d(1-\gamma_1)^n\nonumber\\
&&\le\frac{1}{\alpha_{d-1}}\int_{S^{d-1}} \int_0^{h_K(u)}\int_{(\partial K)^n}\mathbf{1}
\{x_1,\ldots,x_n\in\partial K\setminus C(u,s)\}
\varrho(x_1)\cdots\varrho(x_{n})\nonumber\\
&&\qquad\qquad 
\,\HH^{d-1}(dx_1)\ldots \HH^{d-1}(dx_n)\, ds\, \HH^{d-1}(du) 
+2^d(1-\gamma_1)^n\nonumber\\
&&\le\frac{1}{\alpha_{d-1}}\int_{S^{d-1}} \int_0^{h_K(u)}
\left(1-\int_{\partial K\cap C(u,t)}\varrho(x)\,\HH^{d-1}(dx)\right)^n\,dt
\, \HH^{d-1}(du) 
+2^d(1-\gamma_1)^n.
\label{Cut}
\end{eqnarray}

For suitable positive constants $\gamma_2,\gamma_3,\gamma_4$ we get, 
for $u\in S^{d-1}$ and $t\in(0,\gamma_2)$,
\begin{equation}
\label{lowerloc}
\int_{\partial K\cap C(u,t)}\varrho(x)\,\HH^{d-1}(dx)\;
\left\{
\begin{array}{lcl}
> &\gamma_3t^{d-1},&\mbox{ if $t\in(0,\gamma_2)$},\\[0.5ex]
> &\gamma_4, &\mbox{ if $t\geq\gamma_2$}.
\end{array}
\right.
\end{equation}
In particular, $\gamma_4,\gamma_3(\gamma_2)^{d-1}\in(0,1)$.
We deduce from (\ref{Cut}), \eqref{lowerloc} and \eqref{Gammalimit} 
that, for suitable
$\gamma_5,\ldots,\gamma_9$ with $\gamma_7,\gamma_9\in(0,1)$, 
\begin{eqnarray*}
\E_\varrho (V_1(K)-V_1(K_{n}))&\leq&
\gamma_5\int_0^{\gamma_2} (1-\gamma_3t^{d-1})^n\,dt+\gamma_6\gamma_7^n\\
&=&\gamma_8\int_0^{\gamma_9} s^{\frac1{d-1}-1} \cdot(1-s)^n\,ds+\gamma_6\gamma_7^n
\leq\gamma_{10}n^{\frac{-1}{d-1}}.
\end{eqnarray*}

To prove the lower bound for $\E_\varrho (V_1(K)-V_1(K_{n}))$,
we need the following observation.

\begin{lemma}
\label{circumR}
Let $K\subset\R^d$ be a convex body, and let $h_K$ be twice differentiable 
at $u_0\in S^{d-1}$. Then there is some $R>0$ such that $K\subset x_0-Ru_0+RB^d$, 
where $x_0=\nabla h_K(u_0)\in\partial K$. In particular, 
there exist a measurable set $\Sigma\subset S^{d-1}$ with $\HH^{d-1}(\Sigma)>0$ and 
some $R>0$,
all depending on $K$, such that
for any $u\in \Sigma$ there is some $x\in \partial K$such that $K\subset x-Ru+RB^d$.
\end{lemma}

\begin{proof}
For the proof of the first assertion, we may assume that  $x_0=o$, 
hence also $h_K(u_0)=0$. We put $h:=h_K$. By 
assumption, there is a function $R:\R_+\to[0,\infty)$ with $\lim_{t\to 0^+}R(t)=0$ and
$$
\left|h(u)-\frac{1}{2}\cdot d^2h(u-u_0,u-u_0)\right|\le R(\|u-u_0\|)\|u-u_0\|^2.
$$ 
Thus there is a constant $R_1>0$ and $\delta>0$ such that $h(u)\le R_1\|u-u_0\|^2$ 
for all $u\in S^{d-1}$ with $\langle u,u_0\rangle \ge  1-\delta$. But then for 
$R_2:=\max\{2R_1,\max\{h(u):u\in S^{d-1}\}/(2\delta)\}$ and all $u\in S^{d-1}$, 
we obtain
$$
h(u)\le R_2\left(1-\langle u_0,u\rangle\right)=h(-R_2u_0+R_2B^d,u),
$$
that is $K\subset -R_2u_0+R_2B^d$.

The second assertion follows immediately from the first assertion.
\end{proof}

Let $t_0$ be the inradius of $K$. Now Lemma~\ref{circumR} yields, for $u\in\Sigma$ and $t\in(0,t_0)$, that 
$$
\int_{\partial K\cap C(u,t)}\varrho(x)\,\HH^{d-1}(dx)<\gamma_{11}
\cdot t^{\frac{d-1}2}.
$$
Choosing a constant $\gamma_{12}\in(0,t_0)$ satisfying
$\gamma_{11}(\gamma_{12})^{\frac{d-1}2}<1$, it follows as in the derivation of 
 (\ref{Cut}) that, with a suitable constant $\gamma_{13}\in(0,1)$, we have
\begin{eqnarray*}
\E_\varrho (V_1(K)-V_1(K_{n}))&\geq&\frac{1}{\alpha_{d-1}}\int_{\Sigma}
\int_0^{\gamma_{12}}\left(1-\gamma_{11}t^{\frac{d-1}2}\right)^n\,dt\,\HH^{d-1}(dx)\\
&=&\int_0^{\gamma_{13}} s^{\frac2{d-1}-1}\cdot (1-s)^n\,ds>
\gamma_{14}\cdot n^{\frac{-2}{d-1}}.
\end{eqnarray*}

Theorem~\ref{main} shows that the lower bound of Lemma~\ref{meanwidthbounds}
is of optimal order if $K$ has a rolling ball. In fact, the assumption of a rolling 
ball ensures that the integral on the right side of \eqref{reitzner} is positive. 
This follows, for instance, from the absolute continuity of the Gauss curvature measure 
of a convex body which has a rolling ball (cf.\ \cite{Hug99}).  

On the other hand, the upper bound for  
$\E_\varrho (V_1(K)-V_1(K_{n}))$ is of optimal order
if $K$ is a polytope. To explain this, let $\Sigma_0\subset S^{n-1}$
be contained in the interior of the exterior normal cone
of one of the vertices of $K$ and such that $\HH^{d-1}(\Sigma_0)>0$.
In this case
$$
\int_{\partial K\cap C(u,t)}\varrho(x)\,\HH^{d-1}(dx)<\gamma_{15}\cdot t^{d-1},
$$
for $u\in\Sigma_0$ and $t\in(0,\gamma_{16})$, 
and hence $\E_\varrho (V_1(K)-V_1(K_{n}))\geq \gamma_{17}\cdot n^{\frac{-1}{d-1}}$.

%    Text of article.

%    Bibliographies can be prepared with BibTeX using amsplain,
%    amsalpha, or (for "historical" overviews) natbib style.

\bibliographystyle{amsplain}

\end{document}